\documentclass[english,11pt]{article}

\usepackage[a4paper,bottom=4cm,top=2cm,left=2.54cm,right=2.54cm]{geometry}
\usepackage[english]{babel}
\usepackage[utf8]{inputenc}
\usepackage{url}
\usepackage[usenames]{color}
\usepackage{amsmath}
\usepackage{amssymb}
\usepackage{amsthm}
\usepackage{bbm}
\usepackage{color}
\usepackage{soul}
\usepackage{graphicx}
\usepackage{subcaption}
\usepackage{yfonts}
\usepackage[backend=bibtex,maxnames=4,url=false,style=numeric,isbn=false, giveninits=true,eprint=false,sorting=nyt]{biblatex}

\newcommand{\subdiff}{\partial}
\newcommand{\defeq}{:=}
\DeclareMathOperator*{\argmin}{arg\,min}
\DeclareMathOperator*{\argmax}{arg\,max}

\DeclareMathOperator*{\supp}{supp}
\DeclareMathOperator*{\sign}{sign}
\newcommand{\norm}[1]{\|#1\|}

\newcommand{\abs}[1]{|#1|}

\DeclareMathOperator{\Ext}{Ext}
\DeclareMathOperator{\BV}{BV}
\DeclareMathOperator{\TV}{TV}
\DeclareMathOperator{\TGV}{TGV}
\DeclareMathOperator{\TVLp}{TVL^p}
\DeclareMathOperator{\divergence}{div}
\let\div\relax
\DeclareMathOperator{\div}{\divergence}

\DeclareMathOperator{\dist}{dist}
\DeclareMathOperator{\dom}{dom}

\DeclareMathOperator{\PSNR}{PSNR}
\DeclareMathOperator{\SSIM}{SSIM}
\newcommand{\weakto}{\mathrel{\rightharpoonup}}
\newcommand{\wsto}{\weakto^*}
\newcommand{\R}{\mathbb R}
\newcommand{\eps}{\varepsilon}
\newcommand{\one}{\mathbbm 1}
\newcommand{\charf}{\chi}
\newcommand{\bigO}{O}
\newcommand{\setplus}{\cup}
\let\d\relax
\newcommand{\d}{\partial}
\DeclareMathOperator{\interior}{int}

\renewcommand{\leq}{\leqslant}
\renewcommand{\geq}{\geqslant}
\renewcommand{\phi}{\varphi}

\newcommand{\U}{\mathcal U}

\newcommand{\V}{\mathcal V}
\newcommand{\reg}{\mathcal J}
\newcommand{\Mcal}{\mathcal M}
\newcommand{\Hausdorff}{\mathcal H}
\newcommand{\Der}{D}
\newcommand{\Bnsmu}{B_n^* \mu}
\newcommand{\Bnsmun}{B_n^* \mu_n}
\newcommand{\Mn}{\mathcal M_n^{\eps,C}}
\newcommand{\Lone}{L^1}
\newcommand{\umin}{\tilde u}
\newcommand{\un}{u_n}
\newcommand{\mun}{\mu_n}
\newcommand{\lambdan}{\lambda_n}
\newcommand{\muquer}{\bar \mu}
\newcommand{\nuquer}{\bar \nu}
\newcommand{\lquer}{\bar \lambda}
\newcommand{\uquer}{\bar u}
\newcommand{\uquerJ}{\uquer_\reg}
\newcommand{\mumax}{\tilde \mu}
\newcommand{\lambdamax}{\tilde \lambda}
\newcommand{\mumin}{{\bar \mu}_{min}}
\newcommand{\lmin}{{\bar \lambda}_{min}}
\newcommand{\mut}{\breve \mu}
\newcommand{\lambdat}{\breve \lambda}
\newcommand{\lhat}{\hat \lambda}
\newcommand{\muhat}{\hat \mu}
\newcommand{\phat}{\hat p}
\newcommand{\dJ}{\subdiff \reg}
\newcommand{\exphi}{\bar\phi}
\newcommand{\exB}{\bar B}
\newcommand{\Ent}{{E_n^{(t)}}}
\newcommand{\En}{{E_n}}
\newcommand{\Un}{{\omega_n}}
\newcommand{\uun}{{u|_\Un}}
\newcommand{\Radon}{{\mathfrak M}}
\newcommand{\cone}{{\mathcal K}}
\newcommand{\fs}{{\mathcal S}}

\newtheorem{theorem}{Theorem}
\newtheorem{corollary}[theorem]{Corollary}
\newtheorem{lemma}[theorem]{Lemma}
\newtheorem{proposition}[theorem]{Proposition}

\theoremstyle{definition}

\newtheorem{assumption}{Assumption}
\newtheorem*{assumption*}{Assumption}
\newtheorem{remark}[theorem]{Remark}
\newtheorem*{remark*}{Remark}
\newtheorem*{definition*}{Definition}

\pdfoutput=1
\numberwithin{equation}{section}

\title{Convergence rates and structure of solutions of inverse problems with imperfect forward models}
\author{Martin Burger\footnote{Department Mathematik, University of Erlangen-N\"urnberg, Cauerstrasse~11, 91058 Erlangen, Germany, \mbox{martin.burger@fau.de}} \and Yury Korolev\footnote{Department of Applied Mathematics and Theoretical Physics, University of Cambridge, Wilberforce Road, Cambridge CB3 0WA, UK, \mbox{y.korolev@damtp.cam.ac.uk}} \and Julian Rasch\footnote{Institute for Analysis and Numerics, University of M\"unster, Einsteinstr.~62, 48149 M\"unster, Germany, \mbox{julian.rasch@wwu.de}}}
\date{}

\begin{document}

\maketitle

\begin{abstract}
The goal of this paper is to further develop an approach to inverse problems with imperfect forward operators that is based on partially ordered spaces.
 Studying the dual problem yields useful insights into the convergence of the regularised solutions and allow us to obtain convergence rates in terms of Bregman distances -- as usual in inverse problems, under an additional assumption on the exact solution called the source condition.
 These results are obtained for general absolutely one-homogeneous functionals.
 In the special case of $\TV$-based regularisation we also study the structure of regularised solutions and prove convergence of  their level sets to those of an exact solution.
 Finally, using the developed theory, we adapt the concept of debiasing to inverse problems with imperfect operators and propose an approach to pointwise error estimation in $\TV$-based regularisation.
\end{abstract}


\section{Introduction}
Inverse problems are typically concerned with the interpretation of indirect measurements.
 The measurable data $f$ are typically connected to the quantities of interest $u$ through some forward operator or forward model $A$ that models the data acquisition process. 
 To obtain the quantities of interest $u$ from the data $f$, we need to invert this forward model. 
 Since the inverse of $A$ is typically not continuous, the inversion is ill-posed and one needs to employ regularisation to obtain a stable approximation to $u$.
 Variational regularisation is a common approach to solving ill-posed problems and consists in minimising a weighted sum of a data fidelity term enforcing closeness to the measured data and a regularisation term enforcing some regularity of the reconstructed solution.

In this paper we consider inverse problems in form of an ill-posed operator equation
\begin{equation}\label{Au=f}
Au=f,
\end{equation}
\noindent where $A \colon L^1(\Omega) \to L^\infty(\Omega)$ is a linear operator and $\Omega \subset \R^m$ is a bounded domain.
 We assume that there exists a non-negative solution of~\eqref{Au=f}.

 For an appropriate functional $\reg(\cdot) \colon L^1 \to \R_+ \cup \{\infty\}$ we consider non-negative $\reg$-minimising solutions, which solve the following problem:
\begin{equation}\label{J-min-sol}
\min_{u \in L^1 \colon u \geq 0} \reg(u) \quad \text{s.t. } Au=f.
\end{equation}
\noindent We assume that the feasible set in~\eqref{J-min-sol} has at least one point with a finite value of $\reg$ and denote a (possibly non-unique) solution of~\eqref{J-min-sol} by $\uquerJ$.
Throughout this paper it is assumed that the regularisation functional $\reg(\cdot)$ is convex, proper and absolutely one-homogeneous.
 


In practice the data $f$ are not known precisely and only their perturbed version $\tilde f$ is available. In this case, we cannot simply replace the constraint $A u = f$ in~\eqref{J-min-sol} with $A u = \tilde f$, since the solutions of the original problem~\eqref{Au=f} would no longer be feasible in this case.
 Therefore, we need to relax the equality in~\eqref{J-min-sol} to guarantee the feasibility of solutions of the original problem~\eqref{Au=f}.
 This is the idea of the residual method~\cite{GrasmairHalmeierScherzer2011,IvanovVasinTanana}.
 If the error in the data is bounded by some known constant $\delta$, the residual method accounts to solving the following constrained problem:
\begin{equation}\label{res_meth_1}
\min_{u \in L^1} \reg(u) \quad \text{s.t. } \norm{Au - \tilde f} \leq \delta.
\end{equation}
\noindent The fidelity function becomes in this case the characteristic function of the convex set $\{ u \colon \norm{Au - \tilde f} \leq \delta \}$. In the linear case, the residual method is equivalent to Tikhonov regularisation
\begin{equation}
\min_{u \in L^1} \norm{Au - \tilde f}^2 + \alpha \reg(u)
\end{equation}
\noindent with the regularisation parameter $\alpha  = \alpha(\tilde f, \delta)$ chosen according to Morozov's discrepancy principle~\cite{IvanovVasinTanana}.

In many practical situations not only the data contain errors, but also the forward operator, that generated the data, are not perfectly known. 
 In order to guarantee the feasibility of solutions of the original problem~\eqref{Au=f} in the constrained problem~\eqref{res_meth_1}, one needs to account for the errors in the operator in the feasible set.
 If the errors in the operator are bounded by a known constant $h$ (in the operator norm), the feasible set can be amended as follows in order to guarantee feasibility of the solutions of the original problem~\eqref{Au=f}:
\begin{equation}\label{res_meth_2}
\min_{u \in L^1} \reg(u) \quad \text{s.t. } \norm{\tilde Au - \tilde f} \leq \delta + h \norm{u},
\end{equation}
\noindent where $\tilde A$ is the noisy operator. This optimisation problem is non-convex and therefore presents considerably more computational challenges than its counterpart with the exact operator~\eqref{res_meth_1}.
 Thus, in the context of the residual method, uncertainty in the operator results in a qualitative change in the optimisation problem to be solved, which, in  general, requires using different numerical approaches from those in~\eqref{res_meth_1}.
 The reason for non-convexity is the fact that we used the operator norm to quantify the error in the operator.

An alternative approach was proposed in~\cite{Kor_IP:2014}.
 Instead of the operator norm, it uses intervals in an appropriate partial order to quantify the error in the operator.
 It assumes that, instead of only one instance of approximate data $\tilde f$ and approximate operator $\tilde A$, lower and upper bounds for them are available, i.e. $f^l, f^u \in L^\infty$ and $A^l,A^u \colon L^1 \to L^\infty$ such that
 \begin{equation}\label{bounds}
f^l \leq f \leq f^u, \quad A^l \leq A \leq A^u.
\end{equation}
 
 The first two inequalities are understood in the sense of partial order in $L^\infty$ and the last two in the sense of partial order for linear operators $L^1 \to L^\infty$ (more details on how partial order is defined for linear operators will be given in Section~\ref{Banach_lattices}).


Using the bounds~\eqref{bounds}, the residual method can be reformulated as the following  optimisation problem:
\begin{equation}\label{opt_problem_A}
\min_{u \in L^1 \colon u \geq 0} \reg(u) \quad \text{s.t. } A^l u \leq f^u, \,\, A^u u \geq f^l.
\end{equation} 
\noindent This optimisation problem is convex and has the same structure in the case of errors in the operator as in the error-free case. 
 The fidelity term in this case is the characteristic function of a convex polyhedron. 
 It can be easily verified that any solution of the original problem~\eqref{Au=f} is a feasible solution of~\eqref{opt_problem_A}.
 
 In this paper, we study the dual problem of~\eqref{opt_problem_A}, which can be written as follows
\begin{equation}\label{intro:dual}
\max_{\lambda,\mu \geq 0} \, (\mu,-\phi) \quad \text{s.t. } \lambda-B^* \mu \in \dJ(0),
\end{equation}
\noindent where $(\cdot,\cdot)$ denotes the duality pairing between $L^1$ and $L^\infty$, $\phi = \begin{pmatrix}f^u \\ - f^l \end{pmatrix}$, $B = \begin{pmatrix}A^l \\ - A^u \end{pmatrix}$, $B^*$ is the adjoint of $B$ and $\dJ(0) \subset L^\infty(\Omega)$ is the subdifferential of the regularisation functional at zero. 
 We shall see that, under certain assumptions, $\lambda$ and $\mu$ in~\eqref{intro:dual} are Lagrange multipliers corresponding to the positivity constraint and the constraints $A^l u \leq f^u$ and $A^u u \geq f^l$ in~\eqref{opt_problem_A}, respectively, and $p=\lambda-B^* \mu$ is a subgradient of the regulariser $\reg$ at the optimal solution of~\eqref{opt_problem_A}.
 
 To study the convergence of the minimisers of~\eqref{opt_problem_A} to a solution of~\eqref{Au=f}, we some notion of convergence of the bounds $f^l, f^u$ and $A^l, A^u$ to the exact data $f$ and operator $A$, respectively. 
 For this purpose, we consider sequences of lower and upper bounds $f^l_n, f^u_n$ and $A^l_n, A^u_n$ such that
\begin{eqnarray}\label{bounds_sequence}
&&f^l_n \leq f \leq f^u_n, \quad f^l_{n+1} \geq f^l_n, \quad f^u_{n+1} \leq f^u_n, \label{bounds_f} \\
&&A^l_n \leq A \leq A^u_n, \quad A^l_{n+1} \geq A^l_n, \quad A^u_{n+1} \leq A^u_n \label{bounds_A}
\end{eqnarray}
\noindent and
\begin{equation}
\norm{f^u_n-f^l_n} \to 0, \quad \norm{A^u_n-A^l_n} \to 0.
\end{equation}

With these sequences of bounds, we obtain a sequence of optimisation problems
\begin{equation}\label{opt_problem_A_n}
\min_{u \in L^1 \colon u \geq 0} \reg(u) \quad \text{s.t. } A^l_n u \leq f^u_n, \,\, A^u_n u \geq f^l_n.
\end{equation}
 
It was shown in~\cite{Kor_IP:2014} (see also Theorem~\ref{thm:convergence} in Section~\ref{conv_analysis}) that the minimisers $\un$ of ~\eqref{opt_problem_A_n} converge to $\uquerJ$ as $n \to \infty$.
 In this paper we study this convergence in more detail, ultimately aiming at obtaining convergence rates.
 
 It is well-known~\cite{Benning_Burger_general_fid_2011} that solutions of the dual problem play an important role in establishing convergence rates, therefore we study the behaviour of the dual problem as $n \to \infty$ in more detail. 
 Uncertainty in the operator results in a perturbation of the feasible set of the dual problem~\eqref{intro:dual}. 
 In order to ensure the convergence of its solutions, we would like to know that the solution of the dual problem is stable with respect to such perturbations. 
 Stability theory for optimisation problems with perturbations~\cite{Bonnans_Shapiro:1998} emphasises the role of the so-called Robinson regularity~\cite{Robinson_lin,Robinson_nonlin} in the stability of the solution under perturbations of the feasible set.
 In our particular case a condition on the interior of $\dJ(0)$ (see Assumption~\ref{ass_3}) plays a crucial role in the stability of the dual problem.

Establishing the stability of the dual problem allows us to relate its solutions to solutions of the dual problem in the limit case of exact data and operator, which has a very similar form to~\eqref{intro:dual}:
\begin{equation}\label{intro:dual_limit}
\max_{\lambda,\mu \geq 0} \, (\mu,-\exphi) \quad \text{s.t. } \lambda-\exB^* \mu \in \dJ(0),
\end{equation}
\noindent where $\exphi = \begin{pmatrix}f \\ - f \end{pmatrix}$ and $B = \begin{pmatrix}A \\ - A \end{pmatrix}$.
 If the original problem~\eqref{Au=f} is ill-posed, existence of such limit solutions of the dual problem~\eqref{intro:dual_limit} cannot be guaranteed, unless additional assumptions on the exact solution are made, known as the \emph{source condition}~\cite{Burger_Osher:2004}, which in our case takes the form~\eqref{s.c.}.
Under the source condition we are able to prove uniform boundedness of the Lagrange multipliers and convergence of the subgradient, which allows us to establish convergence rates (Section~\ref{sec:conv_rate}). For the symmetric Bregman distance~\cite{Kiwiel:1997} between the minimisers $\un$ of~\eqref{opt_problem_A_n} and any $\reg$-minimising solution $\uquerJ$ we obtain the following estimate
\begin{equation}
D^{symm}_\reg (\un, \uquerJ) \leq C \cdot \eps_n,
\end{equation}
\noindent where $\norm{f^u_n-f^l_n} = \bigO(\eps_n)$ and $\norm{A^u_n-A^l_n} = \bigO(\eps_n)$.
  These convergence rates coincide with those from~\cite{Benning_Burger_general_fid_2011} for problems with exact operators, providing an interface with existing theory.

We further investigate the solutions of problem~\eqref{opt_problem_A} by studying their geometric properties in the spirit of~\cite{Peyre:2017}.
 In particular, we prove Hausdorff convergence of the level sets of $\TV$-regularised solutions to those of the exact solution. 
 However, unlike the original paper~\cite{Peyre:2017}, we cannot use $\reg(\cdot) = \TV(\cdot)$, since it does not guarantee stability of the dual problem and convergence of the subgradient.
 Instead, we use the full (weighted) $\BV$ norm, choosing $\reg(u) = \TV(u) + \gamma \norm{u}_1$, $\gamma > 0$.

Our numerical experiments with deblurring demonstrate that reconstructions obtained with $\reg(\cdot) = \TV(\cdot) + \gamma\norm{\cdot}_1$, $\gamma>0$, are, indeed, piecewise-constant (if so is the ground truth), while $\reg(\cdot) = \TV(\cdot)$ misses some jumps and results in smoother reconstructions.
 This is surprising, since it contradicts the typical behaviour of $\TV$ in ROF-type models~\cite{ROF}, which is known as staircasing~\cite{Ring:2000,Jalalzai:2016}.
 The reason for this is the additional freedom provided by our constraint-based approach. 
  While in classical ROF-denoising zero is in the subgradient of $\TV$ only if the minimiser is equal to the data (which rarely happens for noisy data), the constraint-based approach allows the subgradient to be zero whenever the noise is small enough and contained within a prescribed corridor around the true data. 
  However, with $\reg(\cdot) = \TV(\cdot) + \gamma\norm{\cdot}_1$, $\gamma>0$, whenever the subgradient of $\reg$ is zero, the subgradient of $\TV$ is equal to $-\gamma$, forcing the reconstruction to be piecewise constant.

Finally, we use the developed theory to adopt the concept of two-step debiasing~\cite{Burger_Rasch_debiasing, Deladelle1, Deladelle2}, which allows to reduce the systematic bias in the reconstruction, such as loss of contrast, to our framework.
 We also propose a method of obtaining asymptotic pointwise lower and upper bounds of $\TV$-regularised solutions in areas, where the exact solution is piecewise-constant.

The paper is organised as follows.
 In Section~\ref{primal_and_dual} we introduce the primal and the dual problems for fixed bounds $f^l$, $f^u$, $A^l$, $A^u$ and study their properties.
 In Section~\ref{conv_analysis} we present the convergence analysis and establish convergence rates.
 In Section~\ref{level_sets} we study geometric properties of $\TV$-regularised solutions and prove Hausdorff convergence of the level sets.
 In Section~\ref{debiasing_and_error_est} we describe our approach to debiasing and pointwise error estimation and in Section~\ref{numerical_experiments} we present the results of our numerical experiments. The Appendices contain some results of more technical nature that we need for the proofs.


\section{Primal and Dual Problems}\label{primal_and_dual}

In order to accurately formulate the primal problem~\eqref{opt_problem_A}, we briefly recall some definitions from the theory of functional spaces with partial order.

\subsection{Banach lattices}\label{Banach_lattices}
$L_p$ spaces, endowed with a partial order relation
\begin{equation*}
f \leq g \text{ iff } f(\cdot) \leq g(\cdot) \text{ a.e.},
\end{equation*}
\noindent become Banach lattices, i.e. partially ordered Banach spaces with well-defined suprema and infima of each pair of elements and a monotone norm~\cite{Schaefer}.
 The set $\{f \in L^p \colon f \geq 0\}$ is called the \emph{positive cone}.
 It can be shown that the interior of the positive cone in an $L^p$ space is empty, unless $p=\infty$~\cite{Schaefer, Schaefer:1958}.

Partial orders in $L^p$ and $L^q$ induce a partial order in  a subspace of the space of linear operators acting from $L^p$ to $L^q$, namely in the space of regular operators.
 A linear operator $A \colon L^p \to L^q$ is called regular, if it can be represented as a difference of two positive operators.
 An operator $A$ is called positive and we write $A \geq 0$ iff $\forall x \in L^p, \,\, x \geq 0 \implies Ax \geq 0$.
 Partial order in the space of regular operators is introduced as follows: $A \geq B$ iff $A-B$ is a positive operator.
  Every regular operator acting between two Banach lattices is continuous~\cite{Schaefer}.

\subsection{Primal and dual problems}

In this section we study in more detail the optimisation problem~\eqref{opt_problem_A} and its dual~\eqref{intro:dual}. For convenience, we repeat problem~\eqref{opt_problem_A} here:
\begin{equation*}
\min_{u \in L^1 \colon u \geq 0} \reg(u) \quad \text{s.t. } A^l u \leq f^u, \,\, A^u u \geq f^l.
\end{equation*}

In order to simplify notation, we introduce 
\begin{equation*}
B = \begin{pmatrix} A^l \\ - A^u \end{pmatrix}
\quad \text{and} \quad 
\exB = \begin{pmatrix}A \\ - A \end{pmatrix}
\end{equation*}
\noindent as well as
\begin{equation*}
\phi = \begin{pmatrix}f^u \\ - f^l \end{pmatrix} 
\quad \text{and} \quad 
\exphi = \begin{pmatrix}f \\ - f \end{pmatrix}.
\end{equation*}
\noindent Obviously, $\phi \in L^\infty \times L^\infty$, however, for the sake of compact notation, we will write $\phi \in L^\infty$, where it will cause no confusion.
 The same holds for the Lagrange multiplier $\mu \in (L^\infty)^* \times (L^\infty)^*$ corresponding to the constraint $B u \leq \phi$, to be introduced later, of which we will simply write $\mu \in (L^\infty)^*$ most of the time.

 With this notation, problem~\eqref{opt_problem_A} can be written as follows
\begin{equation}\label{primal}
\min_{u \in L^1 \colon u \geq 0} \reg(u) \quad \text{s.t. } B u \leq \phi.
\end{equation}
\noindent Denote a (possibly non-unique) minimiser of~\eqref{primal} by $\umin$.

Now let us turn to the dual problem of~\eqref{primal}.
 The Lagrange function is given by the following expression
\begin{equation*}
L(u,\lambda,\mu) = \reg(u) + (\mu,B u - \phi) - (\lambda,u),
\end{equation*}
\noindent where $\lambda \in L^\infty$, $\mu \in (L^\infty)^* \times (L^\infty)^*$, $\lambda,\mu \geq 0$.
 Taking the minimum in $u$, we obtain the following expression for the dual objective:
\begin{equation*}
\min_{u \in L^1} L(u,\lambda,\mu) = -\reg^*(\lambda-B^* \mu) - (\mu,\phi),
\end{equation*}
\noindent where $\reg^*(\cdot)$ is the convex conjugate of $\reg(\cdot)$.
 Since we assumed that $\reg$ is absolutely one-homogeneous, we have that $\reg^*(\cdot)$ is the characteristic function of $\dJ(0)$.
 We discuss the properties of absolutely one-homogeneous regularisation functionals in more detail in Appendix~\ref{app:abs_one_homogeneous}. 
 Hence we obtain the following formulation of the dual problem:
\begin{equation}\label{dual}
\max_{\lambda,\mu \geq 0} \, (\mu,-\exphi) \quad \text{s.t. } \lambda-B^* \mu \in \dJ(0).
\end{equation}


We will mainly consider regularisation functionals as functionals in $L^1$ (and not, for example, in $\BV$), $\dJ(0)$ will therefore be understood as a subset of $L^\infty$ (and not $\BV^*$), with exceptions denoted by a subscript $\dJ_{\V}(0)$, where $\V$ is the corresponding subspace.
 Properties of $\dJ(0)$ for regularisation functionals $J \colon L^1 \to \R_+ \cup \{\infty\}$ of the type $\reg(\cdot) = \TV(\cdot)$ and $\reg(\cdot) = \TV(\cdot)+\gamma\norm{\cdot}_1$, $\gamma>0$, (where $\TV$ may be replaced with a similar regularisation functional, such as $\TGV$) will be discussed in Appendix~\ref{App:dJ(0)}.
 

The following characterisation~\cite{Burger_Gilboa_Moeller_Eckard_spectral} of the the subdifferential of an absolutely one-homogeneous functional will be useful for us later:
\begin{equation}
\dJ(u) = \{p \in L^\infty \colon \reg(v) \geq (p,v) \,\, \forall v \in L^1, \,\, \reg(u) = (p,u)\}.
\end{equation}
In particular, for $u=0$ we get
\begin{equation}
\dJ(0) = \{p \in L^\infty \colon \reg(v) \geq (p,v) \,\, \forall v \in L^1\}.
\end{equation}
\noindent Clearly, the set $\dJ(0)$ is nonempty, convex and closed, although it may be unbounded.

\subsection{Robinson regularity}
We would like to establish strong duality between~\eqref{primal} and~\eqref{dual}. To do this, we need to recall a concept from optimisation theory called \emph{Robinson regularity}.

Consider an optimisation problem
\begin{equation}\label{opt_prob}
\min_{x \in \mathcal C} f(x) \quad \text{s.t. } G(x) \in \cone,
\end{equation}
\noindent where $\mathcal C \subset X$ is a closed and convex set, $G \colon X \to Y$ is continuously Fr\'echet differentiable, $X$ and $Y$ are Banach spaces and $\cone$ is a closed convex subset of $Y$.
 We say that the Robinson regularity condition~\cite{Hinze_Pinnau_Ulbrich} is satisfied at $x_0$ in problem~\eqref{opt_prob} if
\begin{equation}\label{robinson_HPU}
0 \in  \interior(G(x_0) + G_x(x_0)(\mathcal C-x_0) - \cone).
\end{equation}

The next result~\cite[Thm. 4.2]{Bonnans_Shapiro:1998} demonstrates the role that Robinson regularity plays in the existence of the Lagrange multipliers associated with the constraint $G(x) \in \cone$.
\begin{proposition}\label{Robinson_strong_duality}
Suppose that
\begin{itemize}
\item problem~\eqref{opt_prob} is convex;
\item its optimal value is finite;
\item $G(x)$ is continuously differentiable and
\item Robinson regularity condition is satisfied in~\eqref{opt_prob}.
\end{itemize}
Then 
\begin{itemize}
\item strong duality holds between problem~\eqref{opt_prob} and its dual;
\item the set of optimal solutions of the dual problem is non-empty and bounded;
\item if the set of Lagrange multipliers is nonempty for an optimal solution of~\eqref{opt_prob}, then it is the same for any other optimal solution of~\eqref{opt_prob} and coincides with the set of optimal solutions of the dual problem.
\end{itemize}
\end{proposition}

Robinson regularity also plays an important role in the stability of problem~\eqref{opt_prob} under small perturbatuions in $G(\cdot)$. 
 Consider a perturbation of the form $G(x) + u \in \cone$. Denote by $\fs(u)$ the feasible set in the perturbed problem.
 The following result holds~\cite[Prop. 3.3]{Bonnans_Shapiro:1998}.
\begin{proposition}\label{Robinson_stability}
Suppose that Robinson condition~\eqref{robinson_HPU} holds in~\eqref{opt_prob} at $x_0$. Then for every $(x,u)$ in a neighbourhood of $(x_0,0)$ one has
\begin{equation*}
\dist(x,\fs(u)) = \bigO(\dist(G(x)+u,\cone)),
\end{equation*}
\noindent where $\dist(x,\fs(u)) \defeq \inf_{\xi \in \fs(u)} \norm{x-\xi}$ is the distance from $x \in X$ to the set $\fs(u) \subset X$ and $\dist(G(x)+u,\cone) \defeq \inf_{\eta \in \cone} \norm{G(x)+u-\eta}$ is the distance from $G(x)+u \in Y$ to $\cone \subset Y$.
\end{proposition}

This stability result will play an important role in establishing the boundedness of solutions of the dual problem in Sections~\ref{sec:stability_dual}--\ref{sec:boundedness_lambdan}.

\subsection{Relationship between the primal and the dual problems}\label{primal_dual_relationship}

Our aim in this section is to show that Robinson condition~\eqref{robinson_HPU} holds for the primal problem~\eqref{primal}.
This will ensure existence of the Lagrange multipliers and strong duality between~\eqref{primal} and~\eqref{dual}. 

In our case, the function $G$ from~\eqref{opt_prob} is linear, $G(u) = B u  - \phi$, the set $\cone$ is the non-positive cone $\cone_{\leq 0} \defeq \{\psi \in L^\infty \colon \psi \leq 0\}$ and $\mathcal C$ is the non-negative cone $\mathcal C = \{u \in L^1 \colon u \geq 0\}$.
 Since the constraint is linear, the  Robinson condition~\eqref{robinson_HPU} can be written as follows:
\begin{equation}\label{robinson}
0 \in  \interior(\{ B \mathcal C - \phi - \cone_{\leq 0} \}).
\end{equation}

To prove Robinson regularity in problem~\eqref{primal}, we need to assume that $f^l$ and $f^u$ are uniformly bounded away from the true data:
\begin{equation}\label{eps_no_n}
\abs{f^{u,l} - f} \geq \eps \one \quad \text{for some $\eps > 0$}.
\end{equation}
\noindent This assumption will be extended in Assumption~\ref{ass_1} to cover the case of sequences $f^l_n$ and $f^u_n$.

Now we can proceed with the Robinson condition.
\begin{lemma}\label{lemma:robinson_primal}
If~\eqref{eps_no_n} holds then the Robinson condition~\eqref{robinson} is fulfilled at any minimiser $\umin$ of~\eqref{primal}.
\end{lemma}
\begin{proof}
Fix $\eps >0$ and take an arbitrary $\omega \in L^\infty$ with $\norm{w}_\infty < \eps$.
 Our aim is to find $u \geq 0$ and $v \leq 0$ such that $\omega = B u - \phi -v$.

Choose $u = \uquerJ \geq 0$.
 Then we have that
\begin{equation*}
v = B \uquerJ - \phi - \omega \leq \exB \uquerJ - \phi - \omega = \exphi - \phi - \omega \leq -\eps \one - \omega.
\end{equation*}

We see that $v \leq 0$ holds if $\norm{\omega}_\infty$ is small enough\footnote{Note that  since the interior of the positive cone is empty in all $L^p$ spaces except for $L^\infty$, a bound on any other $L^p$ norm of $\omega$ would be insufficient.}, therefore, Robinson condition~\eqref{robinson} holds at $\umin$ in the primal problem~\eqref{primal}.
\end{proof}

Now we are ready to study the relationship between the primal problem~\eqref{primal} and its dual~\eqref{dual}.

\begin{proposition}\label{prop:complementarity}
 Under the assumptions of Lemma~\ref{lemma:robinson_primal}, strong duality between~\eqref{primal} and~\eqref{dual} holds, the complementarity conditions
\begin{equation}\label{complementarity}
\left\{
\begin{aligned}
(\mumax,B \umin - \phi) = 0, \\
(\lambdamax,\umin) = 0
\end{aligned}
\right.
\end{equation}
\noindent are satisfied and $\lambdamax - B^*\mumax \in \dJ(\umin)$, where $(\lambdamax, \mumax)$ denotes the solution of the dual problem~\eqref{dual}.
\end{proposition}
\begin{proof}
Strong duality between the primal problem~\eqref{primal} and its dual~\eqref{dual} follows from Proposition~\ref{Robinson_strong_duality}, since the primal problem~\eqref{primal} is convex, its optimal value is bounded (by $\reg(\uquerJ)$) and Robinson condition~\eqref{robinson} is satisfied.
 Therefore, we have that
\begin{equation}
\reg(\umin) = (\mumax,-\phi).
\end{equation}

Consider the element $\lambdamax - B^*\mumax \in \dJ(0)$.
 Since $\lambdamax - B^*\mumax$ is a subgradient, we get that
\begin{equation*}
\reg(\umin) - (\lambdamax - B^*\mumax ,\umin) \geq 0
\end{equation*}
\noindent and, since $\reg(\umin) = (\mumax,-\phi)$, also that
\begin{equation*}
0 \leq (\lambdamax,\umin) \leq (\mumax,B \umin-\phi) \leq 0
\end{equation*}
\noindent (the latter inequality holds since $\mumax \geq 0 $ and $B \umin \leq \phi$).
 Therefore, the complementarity conditions~\eqref{complementarity} are satisfied.

Since $(\lambdamax - B^*\mumax,\umin) = (\mumax,-B \umin) = (\mumax, - \phi) = \reg(\umin)$ and $\lambdamax - B^*\mumax \in \dJ(0)$, we conclude that $\lambdamax - B^*\mumax \in \dJ(\umin)$ by Proposition~\ref{char_of_dJ(u)}.
\end{proof}


\section{Convergence analysis}\label{conv_analysis}
In this section we turn our attention to sequences of primal and dual problems defined using sequences of bounds~\eqref{bounds_sequence}:
\begin{equation}\label{primal_n}
\min_{u \in L^1 \colon u \geq 0} \reg(u) \quad \text{s.t. } B_n u = \phi_n.
\end{equation}
\noindent and
\begin{equation}\label{dual_n}
\max_{\lambda,\mu \geq 0} \, (\mu,-\phi_n) \quad \text{s.t. } \lambda- \Bnsmu \in \dJ(0).
\end{equation}

We will be particularly interested  in the convergence of their solutions to those of the limit problems with exact data and operator (note that~\eqref{primal_limit} is just another way of writing~\eqref{J-min-sol}):
\begin{equation}\label{primal_limit}
\min_{u \in L^1 \colon u \geq 0} \reg(u) \quad \text{s.t. } \exB u = \exphi.
\end{equation}
and
\begin{equation}\label{dual_limit}
\max_{\lambda,\mu \geq 0} \, (\mu,-\exphi) \quad \text{s.t. } \lambda-\exB^* \mu \in \dJ(0).
\end{equation}

We start with the convergence of primal variables $\un$ -- solutions of~\eqref{primal_n}.

\subsection{Convergence of primal solutions}
It can be easily verified that any $\reg$-minimising solution $\uquerJ$ satisfies $B_n \uquerJ \leq \phi_n$ for all $n$, which implies $\reg(\un) \leq \reg(\uquerJ)$.
 It has been shown in~\cite{Kor_IP:2014} that under standard assumptions on $\reg$ the minimisers of~\eqref{primal_n} converge to a $\reg$-minimising solution $\uquerJ$ strongly in $L^1$: 
\begin{theorem}\label{thm:convergence}
If the regulariser $\reg(\cdot) \colon L^1 \to \R_+ \cup \{\infty\}$
\begin{itemize}
\item is strongly lower-semicontinuous in $L^1$,
\item its non-empty sub-levelsets $\{ u \colon \reg(u) \leq C\}$ are strongly sequentially compact,
\end{itemize}
\noindent then there exists a minimiser $\un$ of~\eqref{primal}, $u_n \to \uquerJ$ strongly in $L^1$ (possibly, along a subsequence) and $\reg(\un) \to \reg(\uquerJ)$.
\end{theorem}
\noindent The proof is similar to that in~\cite[Thm 2]{Kor_IP:2014}.

Assumptions of Theorem~\ref{thm:convergence} are satisfied, for example, for the (weighted) $\BV$- norm $\reg(u) = \TV(u) + \gamma\norm{u}_1$, $\gamma > 0$,  or its topological equivalents with $\TV$ replaced with, e.g., $\TGV$~\cite{bredies2009tgv} or $\TVLp$~\cite{Burger_TVLp_2016}.
 The term $\gamma\norm{u}_1$ can be dropped if its boundedness is implied by the condition $B_n u \leq \phi_n$ (we will see an example of this in Section~\ref{boundedness_of_u}).
 
 In order to make sure the Robinson condition is satisfied in~\eqref{primal_n} for all $n$, we need to extend the assumption that we already made in~\eqref{eps_no_n} to sequences of bounds $f^l_n$ and $f^u_n$.
 In order to have all assumptions on convergence in one place, we also include our assumptions on the convergence of the operator, which we will need later, in the following
\begin{assumption}\label{ass_1}
Suppose that there exists a sequence $\eps_n \downarrow 0$ and a constant $C_0 \geq 1$ as well as a sequence $\eta_n \downarrow 0$ and a constant $D_0 > 0$ such that
\begin{eqnarray}
&&\eps_n \one \leq \phi_n - \exphi \leq C_0 \cdot \eps_n \one, \label{conv_phin} \\
&&\norm{B-B_n}_{L^1 \to L^\infty} \leq \eta_n, \label{conv_Bn} \\
&&\limsup_{n \to \infty} \frac{\eta_n}{\eps_n} \leq D_0 \label{etan/epsn}.
\end{eqnarray}
\end{assumption}

The meaning of~\eqref{conv_phin} is that $\phi_n$ converges to $\exphi$ uniformly, but not too fast; the difference is always uniformly bounded away from zero.
 The second inequality in~\eqref{conv_phin} obviously implies that $\norm{\phi_n-\exphi}_\infty = \bigO(\eps_n)$.
The meaning of~\eqref{etan/epsn} is that the data do not converge faster than the operator.
 
\subsection{Boundedness of feasible solutions of the primal problem}\label{boundedness_of_u}
In this section we will show that under some assumptions about the exact forward operator $A$ all elements of the feasible set $\{u \geq 0 \colon B_n u \leq \phi_n\}$ are uniformly bounded in $L^1$.
 Assumptions from this section will not be used in the rest of the paper, unless specifically stated, and the results of other sections will be also valid for more general forward operators.

Since all elements $u$ of the feasible set $\{u \geq 0 \colon B_n u \leq \phi_n\}$ are positive, we have that $\norm{u}_1 = (u,\one)$.
 Consider the following optimisation problem:
\begin{equation}\label{boundedness_primal}
\max_{u \geq 0} \, (u,\one) \quad \text{s.t. } B_n u \leq \phi_n.
\end{equation}

It is a linear programming problem and its dual is as follows~\cite{Anderson_Nash_LP_inf_dim}
\begin{equation}\label{boundedness_dual}
\min_{\mu \geq 0} \, (\mu,\phi_n) \quad \text{s.t. } \Bnsmu \geq \one.
\end{equation}

We make the following assumption about the exact forward operator $A$:
\begin{assumption}\label{ass_5}
Assume that the adjoint operator $A^* \colon (L^\infty)^* \to L^\infty$ satisfies the following condition:
\begin{equation*}
A^* \one \geq c \one
\end{equation*}
\noindent for some constant $c>0$.
\end{assumption}

This assumption is satisfied in many imaging inverse problems, such as deconvolution~\cite{Burger_Osher_TV_Zoo} and PET~\cite{Sawatzky:2013}. It also trivially satisfied for denoising  and inpainting.

\paragraph{The case $B_n \equiv \exB$.} In order to get some intuition, let us first consider the case $B_n \equiv \exB$.
 
\begin{theorem}
Suppose that $B_n \equiv \exB$ for all $n$ and Assumption~\ref{ass_5} is satisfied.
 Then all elements of the feasible set in~\eqref{primal_n} are uniformly bounded in $L^1$.
\end{theorem}
\begin{proof}
It is easy to verify that $\mu = \frac{1}{c}(\one, 0)$ is a feasible solution of~\eqref{boundedness_dual} (here $c$ is the constant from Assumption~\ref{ass_5}).
 Indeed, we have that $\exB^* \mu = \frac{1}{c} A^* \one - 0 \geq \one$.
 By weak duality we have that problem~\eqref{boundedness_primal} is bounded and $\norm{u}_1 \leq \frac{1}{c}(f^u_n,\one) \leq C$, since $f^u_n \to f$ strongly in $L^\infty$.
\end{proof}

\paragraph{The general case.} In the general case we obtain a similar result using Assumption~\ref{ass_1}.
\begin{theorem}
Suppose that~\eqref{conv_Bn} holds and Assumption~\ref{ass_5} is satisfied.
 Then all elements of the feasible set in~\eqref{primal} are uniformly bounded in $L^1$.
\end{theorem}
\begin{proof}
\eqref{conv_Bn} implies that $\norm{(A^u_n)^* - (A^l_n)^*}_{(L^\infty)^* \to L^\infty} \leq \eta_n$ and $\norm{(A^u_n)^*\mu - (A^l_n)^*\mu}_\infty \leq \eta_n \norm{\mu}_1$ for any $\mu \in (L^\infty)^*$.
 Therefore, we have that
\begin{equation*}
-\eta_n \norm{\mu}_1 \one \leq (A^u_n)^* \mu - (A^l_n)^* \mu \leq \eta_n \norm{\mu}_1 \one
\end{equation*}
\noindent and
\begin{equation*}
(A^l_n)^* \mu \geq (A^u_n)^* \mu - \eta_n \norm{\mu}_1 \one.
\end{equation*}
\noindent Taking $\mu =\one$, we get that
\begin{equation*}
(A^l_n)^* \one \geq (A^u_n)^* \one - \eta_n \norm{\one}_1 \one \geq A^* \one - \eta_n \abs{\Omega} \one \geq (c - \abs{\Omega} \eta_n) \one.
\end{equation*}

\noindent Now consider $\mu = \frac{2}{c}(\one,0)$.
 It is a feasible solution of problem~\eqref{boundedness_dual}, since 
\begin{equation*}
\Bnsmu = \frac{2}{c}(A^l_n)^*\one - 0 \geq (2 - \frac{2}{c} \abs{\Omega} \eta_n) \one \geq \one
\end{equation*}
\noindent if $n$ is large enough, since $\norm{\one}_1 = |\Omega| < \infty$.
 Therefore, problem~\eqref{boundedness_primal} is bounded and $\norm{u}_1 \leq \frac{2}{c}(f^u_n,\one) \leq C$.
\end{proof}
 
 \subsection{Strong duality in the limit case}

Since the exact operator $\exB$ is ill-posed, we cannot expect Robinson regularity to hold in the primal limit problem~\eqref{primal_limit} and, therefore, we cannot guarantee strong duality between~\eqref{primal_limit} and~\eqref{dual_limit} or even the existence of solutions of the dual limit problem~\eqref{dual_limit}, let alone its stability and convergence of the solutions of~\eqref{dual_n}. 
 As usual in ill-posed problems, we will need to make an additional assumption about the dual limit problem, called the \emph{source condition}~\cite{Burger_Osher:2004}, which in our case is can be written as follows:
\begin{assumption}[Source condition]\label{ass_2}
Assume that $\exists \muquer \geq 0$ such that 
\begin{equation}\label{s.c.}
-\exB^*\muquer \in \dJ(\uquerJ).
\end{equation}
\end{assumption}

Let us note that since $\muquer = (\muquer_1,\muquer_2)$, where $\muquer_{1,2} \in (L^\infty)^*_+$, $-\exB^*\muquer = A^*\muquer_2 - A^*\muquer_1 = A^*(\muquer_2 - \muquer_1) = A^*\nuquer$ with $\nuquer \in (L^\infty)^*$.
 Therefore, \eqref{s.c.}~implies the source condition from~\cite{Burger_Osher:2004}.
 On the other hand, since every element $\nu \in (L^\infty)^*$ can be represented as a difference of two positive elements $\nu = \nu_+ - \nu_-$~\cite{Schaefer}, the source condition from~\cite{Burger_Osher:2004} also implies~\eqref{s.c.}.

\begin{proposition}
Under Assumption~\ref{ass_2}, the pair $(0,\muquer)$ solves the dual problem~\eqref{dual_limit} and strong duality between~\eqref{primal_limit} and~\eqref{dual_limit} holds.
\end{proposition}
\begin{proof}
Since $\reg$ is absolute one-homogeneous, the source condition~\eqref{s.c.} implies that $(\muquer, -\exphi) = (\muquer,-\exB\uquerJ) = (-\exB^* \muquer,\uquerJ) = \reg(\uquerJ)$ (see Proposition~\ref{prop:one_homog} in Appendix~\ref{app:abs_one_homogeneous}). 
 On the other hand, for any feasible $\mu$, $(\mu,-\exphi) \leq \reg(\uquerJ)$ by weak duality.
 Therefore, the pair $(0,\muquer)$ solves~\eqref{dual_limit} and strong duality holds.
\end{proof}

\subsection{Stability of the dual problem} \label{sec:stability_dual}
 
The goal of this section is to show that the feasible set in the dual limit problem~\eqref{dual_limit} is stable under perturbations of the following form:
\begin{equation*}
\lambda - \exB^*\mu + r \in \dJ(0)
\end{equation*}
\noindent for some small $r \in L^\infty$.
 Denote by $\fs(r)$ the feasible set of the perturbed problem.
 From Proposition~\ref{Robinson_stability} we know that if Robinson condition holds at some point $(\lambda_0, \mu_0)$ at $r=0$ then $\dist((\lambda_0, \mu_0),\fs(r)) = \bigO(\dist(\lambda_0 - \exB^*\mu_0 + r, \dJ(0)))$.

In order to show that Robinson condition~\eqref{robinson_HPU} is satisfied in~\eqref{dual_limit} at  $(0,\muquer)$, we need to make the following
\begin{assumption}\label{ass_3}
Assume that
\begin{equation*}
0 \in \interior_{L^\infty} \dJ(0).
\end{equation*}
\end{assumption}

We emphasise that, since we consider the regularisation functional in $L^1$, its subdifferential at zero should be considered in $L^\infty$, rather than, for instance, $\BV^*$.
Assumption~\ref{ass_3} holds, for example, for the (weighted) $\BV$ norm $\reg(\cdot) = \TV(\cdot) + \gamma\norm{\cdot}_1$, $\gamma>0$, also in the case when it is considered as a functional from $L^1$ to $\R \cup \{\infty\}$ and not from $\BV$ to $\R$, as shown in Appendix~\ref{App:dJ(0)}.
 Assumption~\ref{ass_3} fails, however, for $\reg(\cdot) = \TV(\cdot)$ (see Appendix~\ref{App:dJ(0)} as well).
 

\begin{lemma}\label{robinson_dual}
Under Assumption~\ref{ass_3} Robinson condition~\eqref{robinson_HPU} holds in~\eqref{dual_limit}.
\end{lemma}
\begin{proof}
 We need to show that
\begin{equation*}
0 \in \interior(\{ \lambda - \exB^*\mu - p \, | \,\lambda, \mu \geq 0, \, p \in \dJ(0)\})
\end{equation*}

For this, we need to show that for an arbitrary $q \in L^\infty$, $\norm{q} < \eps$, we have that $q \in \{ \lambda - \exB^*\mu - p \, | \,\lambda, \mu \geq 0, \, p \in \dJ(0)\}$, if $\eps$ is small enough. Fix some $q \in L^\infty$, $\norm{q} < \eps$.
 We need to find $\lambda, \mu \geq0$ and $p \in \dJ(0)$ such that
\begin{equation*}
q = \lambda - \exB^*\mu - p.
\end{equation*}

The required condition is satisfied if we take $\lambda = \mu = 0$ and $p=-q$.
 Since $\norm{p} = \norm{q} < \eps$ and $0 \in \interior \dJ(0)$, we see that $p \in \dJ(0)$ and Robinson condition is satisfied.
\end{proof}

\subsection{Boundedness of Lagrange multipliers $\mun$}\label{boundedness_of_mun}
Now we want to investigate the convergence of the Largange multipliers $\mun$ and $\lambdan$, which by the results of Section~\ref{primal_dual_relationship} are related to the subgradient of $\reg$ at $\un$. As noted earlier, convergence of the subgradient plays an important role in establishing convergence rates.

\paragraph{The case $B_n \equiv \exB$.} Again, we will first consider the case $B_n \equiv \exB$.
 We will see that it significantly differs from the general case, because it does not require Assumption~\ref{ass_3}.
\begin{theorem} Suppose that $B_n \equiv B$ for all $n$ and Assumptions~\ref{ass_1} (convergence) and~\ref{ass_2} (source condition) hold.
 Then
\begin{equation}\label{mun_bounded_1}
\norm{\mun}_1 \leq C_0 \norm{\muquer}_1.
\end{equation}
\end{theorem}
\begin{proof}
Consider two problems~\eqref{dual_n} and~\eqref{dual_limit}.
 Since $B_n \equiv \exB$, their feasible sets coincide and $(0,\muquer)$ is a feasible solution of~\eqref{dual_n}.
 Therefore, $(\muquer,-\phi_n) \leq (\mun,-\phi_n)$.
 Similarly, $(\lambdan,\mun)$ is a feasible solution of~\eqref{dual_limit} and $(\mun,-\exphi) \leq (\muquer,-\exphi)$.
 Combining these two estimates, we conclude that $(\mun,\phi_n-\exphi) \leq (\muquer, \phi_n-\exphi)$.
 Assumption~\ref{ass_1} implies that
\begin{equation*}
\eps_n (\mun,\one) \leq (\mun,\phi_n-\exphi) \leq C_0 \cdot \eps_n (\muquer, \one).
\end{equation*}
\noindent Since $\mun \geq 0$ and $\muquer \geq 0$, estimate~\eqref{mun_bounded_1} follows.
\end{proof}


\paragraph{The general case.} In the general case the optimal solution of one problem is no longer a feasible solution of the other one, but due to the stability of the feasible set, there are feasible points ``not too far away".

\begin{theorem}\label{thm:mun_bounded} 
Suppose that  Assumptions~\ref{ass_1} (convergence),~\ref{ass_2} (source condition) and~\ref{ass_3} (non-empty interior) hold.
 Then there exists a constant $\tilde C \geq C_0$ such that
\begin{equation}\label{mun_bounded_2}
\norm{\mun}_1 \leq \tilde C \norm{\muquer}_1.
\end{equation}
\end{theorem}
\begin{proof}
Consider first the limit problem~\eqref{dual_limit}.
 Since Robinson condition holds in the dual problem (Lemma~\ref{robinson_dual}), by Proposition~\ref{Robinson_stability} we have that
\begin{equation*}
\dist((0,\muquer), \{\lambda,\mu \geq 0 \colon \lambda - \Bnsmu \in \dJ(0)\}) \leq C_1 \dist(-B_n^*\muquer, \dJ(0)) \leq C_1 \norm{\exB-B_n} \norm{\muquer}_1,
\end{equation*}
\noindent where the last inequality holds because $-B_n^*\muquer = -\exB^*\muquer + (\exB^*\muquer - B_n^*\muquer)$ and $-\exB^*\muquer \in \dJ(0)$.

Therefore, there exist $\lambdat, \mut \geq 0$ such that $\lambdat - B_n^*\mut \in \dJ(0)$ and $\norm{(\lambdat, \mut) - (0,\muquer)} \leq C_1 \norm{\exB-B_n} \norm{\muquer}_1$ (and, therefore, $\norm{\mut-\muquer}_1 \leq C_1 \norm{\exB-B_n} \norm{\muquer}_1$).
 Since $(\lambdat, \mut)$ is feasible, we get that $(\mut, -\phi_n) \leq (\mun,-\phi_n)$.
 Furthermore, since $\abs{(\muquer-\mut,\phi_n)} \leq \norm{\muquer-\mut}\norm{\phi_n} \leq \tilde C_1 \norm{\exB-B_n} \norm{\muquer}_1$, where $\tilde C_1 = C_1 \norm{\phi_n}$ and can be chosen arbitrary close to $C_1 \norm{\exphi}$, we get that
\begin{equation}\label{mu_n:est_1}
(\muquer,-\phi_n) \leq (\mun,-\phi_n) + \tilde C_1 \norm{\exB-B_n} \norm{\muquer}_1.
\end{equation}

Similarly, in the dual problem for finite $n$~\eqref{dual_limit} we get that there exist $\lambdat_n, \mut_n \geq 0$ such that $\lambdat_n - \exB^*\mut_n \in \dJ(0)$ and $\norm{\mut_n-\mun} \leq \tilde C_2 \norm{\exB-B_n} \norm{\mun}_1$.
 Since $(\lambdat_n, \mut_n)$ is feasible, we get that $(\mut_n,-\exphi) \leq (\muquer,-\exphi)$ and
\begin{equation}\label{mu_n:est_2}
(\mun, -\exphi) \leq (\muquer, -\exphi) + \tilde C_2 \norm{\exB-B_n} \norm{\mun}_1.
\end{equation}

Combining~\eqref{mu_n:est_1} and~\eqref{mu_n:est_2}, we get the following estimate
\begin{multline*}
\eps_n \norm{\mun}_1 \leq (\mun,\phi_n-\exphi) \leq (\muquer,\phi_n-\exphi) + \tilde C_1 \norm{\exB-B_n} \norm{\muquer}_1 + \tilde C_2 \norm{\exB-B_n} \norm{\mun}_1 \leq \\
C_0\cdot \eps_n \norm{\muquer}_1 + \tilde C_1 \eta_n \norm{\muquer}_1 + \tilde C_2 \eta_n \norm{\mun}_1,
\end{multline*}
\noindent or, equivalently,
\begin{equation*}
\norm{\mun}_1 \left( 1 - \tilde C_2 \frac{\eta_n}{\eps_n} \right) \leq C_0\norm{\muquer}_1 \left(1 + \frac{\tilde C_1}{C_0}\frac{\eta_n}{\eps_n}\right).
\end{equation*}
\noindent If the constant $D_0$ in~\eqref{etan/epsn} is small enough, this implies~\eqref{mun_bounded_2} with
\begin{equation}\label{eq:tilde C}
\tilde C = C_0 \frac{\left(1 + \frac{\tilde C_1}{C_0}D_0\right)}{\left( 1 - \tilde C_2 D_0 \right)}.
\end{equation}
\end{proof}

%

\begin{corollary}
Since the sequence $\{\mun\}$ is bounded in $(L^\infty)^*$, the sequence $\{\Bnsmun\}$ is bounded in $L^\infty$ if the operators $B_n^*$ are bounded from $(L^\infty)^*$ to $L^\infty$.
\end{corollary}

\begin{remark}
Note that in the case $B_n \equiv \exB$ we did not use the stability of the dual problem and, therefore, did not need Assumption~\ref{ass_3} to show that the Lagrange multipliers $\mun$ are bounded.
 This demonstrates that Assumption~\ref{ass_3} plays an important role specifically in problems with an imperfect operator.
\end{remark}

\subsection{Boundedness of Lagrange multipliers $\lambdan$} \label{sec:boundedness_lambdan}
\begin{proposition}
Suppose that $\reg(\one) < + \infty$.
 Then under the assumptions of Theorem~\ref{thm:mun_bounded} we have that  $\norm{\lambdan}_1 \leq C$.
\end{proposition}
\begin{proof}
Since $\lambdan - \Bnsmun \in \dJ(0)$, we have that $\forall n$ $\reg(u) - (\lambda_n - \Bnsmun,u) \geq 0$, or, equivalently,
\begin{equation*}
(\lambdan,u) \leq \reg(u) + (\mun,B_n u).
\end{equation*}
\noindent Choosing $u=\one$, we obtain an estimate of the $L^1$ norm of $\lambdan$ (since $\lambdan \geq 0$):
\begin{equation}\label{est_lambda}
\norm{\lambdan}_1 = (\lambdan,\one) \leq \reg(\one) + (\mun,B_n \one) \leq \reg(\one) + (\mun,\exB \one) \leq \reg(\one) + \norm{\mun}_1\norm{\exB \one}_\infty \leq C.
\end{equation}
\end{proof}

It is worth noting that, although $\lambdan \in L^\infty$, we only get a bound on the $L^1$ norm of $\lambdan$ here.
 

\subsection{Convergence of the subgradient}
 
 Now we are ready to study the convergence of the subgradient of $\reg$ at the optimal solution of the primal problem $\un$. 
 
\begin{proposition} \label{prop:mun_conv}
 Under the assumptions of Theorem~\ref{thm:mun_bounded} the sequence $\mun$ has a weakly-$^*$ convergent subsequence (in $(L^\infty)^*$), which we still denote by $\mun$, $\mun \wsto \muhat$, and $\Bnsmun \wsto \exB^* \muhat$ in $L^\infty$.
\end{proposition}
\begin{proof}
Since $\mun$'s are bounded in $(L^\infty)^*$, $\mun \wsto \muhat$ (along a subsequence) in $(L^\infty)^*$ by the Banach-Alaoglu theorem~\cite{Burger_Osher_TV_Zoo}, i.e. for any $\phi \in L^\infty$ we have that $(\mun,\phi) \to (\muhat, \phi)$.
 Considering, for an arbitrary $u \in L^1$, the scalar product $(\Bnsmun,u)$, we note that $(\Bnsmun,u) = (\mun,B_n u) \to (\muhat,\exB u) = (\exB^*\muhat, u)$, since $\mun \wsto \muhat$ in $(L^\infty)^*$ and $B_n u \to \exB u$ in $L^\infty$.
 Therefore, $\Bnsmun \wsto \exB^* \muhat$ in $L^\infty$.
\end{proof}

\begin{remark}\label{min_norm_certificate}
Since~\eqref{mun_bounded_2} holds for any $\muquer$ delivering the source condition, i.e. every $\muquer \geq 0$ such that $\exists \lquer \geq 0$, $\lquer - \exB^*\muquer \in \dJ(\uquerJ)$, it also holds for the (possibly non-unique) minimum-norm certificate that solves the following problem:
\begin{equation}\label{eq:mumin}
(\lmin,\mumin) \in \argmin_{\lambda,\mu \geq 0} \norm{\mu}_1 \quad \text{s.t. } \lambda - \exB^* \mu \in \dJ(\uquerJ).
\end{equation}
\noindent Since $\mun \wsto \muhat$ in $L^\infty$, we have that $\norm{\muhat}_1 \leq \tilde C \norm{\mumin}_1$, where $\tilde C$ is given by~\eqref{eq:tilde C}.
 If the operator converges faster than the data, i.e $\lim_{n \to\infty} \frac{\eta_n}{\eps_n} = 0$, then, taking the limit in~\eqref{eq:tilde C} and letting $D_0 \to 0$, we get that $\norm{\muhat}_1 \leq C_0\norm{\mumin}_1$.
 On the other hand, since $\lhat - \exB^* \muhat \in \dJ(\uquerJ)$, we have that $\norm{\mumin}_1 \leq \norm{\muhat}_1$.
 If $C_0=1$, we have that $\norm{\muhat}_1 = \norm{\mumin}_1$, i.e. $\muhat$ is a minimum-norm certificate.
 In the general case~\eqref{conv_Bn} we can only say that the norm of $\muhat$ is bounded by that of $\muquer$ times a constant.
\end{remark}

We would like to have that the whole subgradient $p_n = \lambda_n - \Bnsmun$ is bounded is $L^\infty$, however, we have only a bound in $L^1$ for the first summand $\lambda_n$.
 However, we know that $(\lambda_n,\un)=0$ for all $n$ (cf.~\eqref{complementarity}) and, as we shall see later, the same holds for the characteristic functions of the level sets of $\un$ (see Section~\ref{level_sets}).
 Therefore, boundedness (in $L^\infty$) of $\Bnsmun$ will suffice in most cases.

To study the convergence of the subgradient $p_n$, let us consider the subspace $\V = \{u \in L^1 \colon \reg(u) < +\infty\}$.


 \begin{theorem} \label{thm:conv_subgrad_BV*}
 Suppose that $\reg(\cdot)$ is a norm on $\V = \{u \in L^1 \colon \reg(u) < +\infty\}$.
 Then, under the assumptions of Theorem~\ref{thm:mun_bounded}, we have that
 \begin{equation*}
 p_n = \lambda_n - \Bnsmu_n \wsto \phat \quad \text{in } \V^*
\end{equation*}
\noindent and $\phat \in \dJ_{\V^*}(\uquerJ)$ for all $\reg$-minimising solutions $\uquerJ$.
 We also have that $\lambda_n \wsto \lhat$ in $\V^*$, $(\lhat,u) \geq 0$ for any $u \in \V$ such that $u \geq 0$ and $(\lhat,\uquerJ) = 0$.
 All convergences are along a subsequence, which we do not relabel.
 \end{theorem}
 \begin{proof}
Since the dual of a norm is the characteristic function of the unit ball in the dual norm~\cite{Borwein_Zhu}, we have that 
\begin{equation*}
\dJ_\V(0) = \{p \in \V^* \colon \norm{p}_{\V^*} \leq 1\}.
\end{equation*}
\noindent By the Banach-Alaoglu theorem we get weak-$^*$ convergence of a subsequence $p_n \wsto \phat$ in $\V^*$.
 Weak-$^*$ convergence $\Bnsmun \wsto \exB^* \muhat$ in $L^\infty$ (and, therefore, in $\V^*$) implies that $\lambda_n \wsto \lhat = \phat - \exB^* \muhat$ in $\V^*$.

To study the properties of $\phat$, we make the following observation:
\begin{eqnarray*}
&&(p_n, \un) = (\lambda_n, \un) - (\Bnsmun, \un) = (\mun, -B_n \un) = (\mun, -\phi_n), \\
&&(p_n,\uquerJ) = (\lambda_n, \uquerJ) - (\Bnsmun, \uquerJ) \geq (\mun,-B_n \uquerJ) \geq (\mun, -\phi_n) = (p_n, \un)
\end{eqnarray*}
\noindent for any $\reg$-minimising solution $\uquerJ$. 
 (Note that the term $(\lambda_n, \un)$ in the first line vanishes by Proposition~\ref{prop:complementarity}).
 Therefore, we have that
\begin{equation*}
0 \leq \reg(\uquerJ) - (p_n,\uquerJ) \leq  \reg(\uquerJ) - (p_n,\un) =  \reg(\uquerJ) - \reg(\un) \to 0.
\end{equation*}
\noindent Hence, we get that $(p_n,\uquerJ) \to \reg(\uquerJ)$.
 Combining this with $(p_n,\uquerJ) \to (\phat,\uquerJ)$ (since $\uquerJ \in \V$), we get that $(\phat,\uquerJ) = \reg(\uquerJ)$ and $\phat \in \dJ_{\V^*}(\uquerJ)$ (the condition $\phat \in \dJ_{\V^*}(0)$ follows from weak-$^*$ closedness of the unit ball in $\V^*$).

Clearly, $(\lhat,u) \geq 0$ for all $u \in \V$, $u \geq 0$.
 Noting that 
\begin{equation}
\reg(\un) = (p_n,\un) = (\lambda_n,\un) - (\Bnsmun,\un) = (-\Bnsmun,\un) \to (-\exB^* \muhat, \uquerJ),
\end{equation}
\noindent we conclude that $(-\exB^* \muhat, \uquerJ) = \reg(\uquerJ)$.
 Combining this with $(\phat,\uquerJ) = \reg(\uquerJ)$, we get that $(\lhat,\uquerJ) = 0$.
\end{proof}

\subsection{Convergence rates}\label{sec:conv_rate}

The results of the previous sections allow us to obtain convergence rates of $\un \to \uquerJ$ in terms of the (generalised) Bregman distance~\cite{Burger_Osher:2004}.
 Indeed, consider the symmetric Bregman distance, which for absolutely one-homogeneous functionals can be written as follows
\begin{equation*}
D_\reg^{symm}(\uquerJ,\un) = D_\reg^{p_n}(\uquerJ,\un) + D_\reg^{\phat}(\un,\uquerJ) = (p_n - \phat, \un - \uquerJ).
\end{equation*}

\begin{theorem}
Under the assumptions of Theorem~\ref{thm:mun_bounded} the following estimate holds for any $\reg$-minimising solution $\uquerJ$:
\begin{equation}\label{conv_rate}
D_\reg^{symm}(\uquerJ,\un) \leq C \norm{\muquer}_1 \cdot \eps_n.
\end{equation}
\end{theorem}
\begin{proof}
We obtain the following estimate for the symmetric Bregman distance:
\begin{align*}
(p_n - \phat, \un - \uquerJ) &= (\lambda_n,\un) - (\lambda_n,\uquerJ) - (\lhat,\un) + (\lhat,\uquerJ) \\
 & - (\mun,B_n \un) + (\mun, B_n \uquerJ) + (\muhat, \exB \un) - (\muhat, \exB \uquerJ)  \\
 &\leq - (\mun,B_n \un) + (\mun, B_n \uquerJ) + (\muhat, \exB \un) - (\muhat, \exB \uquerJ),
 \end{align*}
 \noindent since $(\lambda_n,\un) = 0$ (Proposition~\ref{prop:complementarity}), $(\lambda_n,\uquerJ) \geq 0$, $(\lhat,\un) \geq 0$, $(\lhat,\uquerJ) = 0$ (Theorem~\ref{thm:conv_subgrad_BV*}). Using the fact that $(\mun,B_n \un) = (\mun,\phi_n)$ and $B_n \uquerJ \leq \phi_n$, we note that
 \begin{equation*}
 - (\mun,B_n \un) + (\mun, B_n \uquerJ) = (\mun, B_n \uquerJ - \phi_n) \leq 0
 \end{equation*} 
 \noindent and, therefore,
 \begin{align*}
(p_n - \phat, \un - \uquerJ) & \leq (\muhat, \exB \un - \exphi) = (\muhat, B_n \un - \exphi) + (\muhat, (\exB - B_n) \un) \\
& \leq (\muhat, \phi_n - \exphi) + (\muhat, (\exB - B_n) \un).
\end{align*}
\noindent The last inequality is due to the fact that $B_n \un \leq \phi_n$. Using Assumption~\ref{ass_1} and the fact that $\norm{\muhat}_1 \leq C\norm{\muquer}_1$ (Theorem~\ref{thm:mun_bounded} and Proposition~\ref{prop:mun_conv}), we finally obtain the required estimate
\begin{align*}
(p_n - \phat, \un - \uquerJ) & \leq (\muhat, \phi_n - \exphi) + (\muhat, (\exB - B_n) \un) \leq C \cdot \eps_n \cdot (\muhat,\one) + \norm{\muhat}_1 \norm{\exB-B_n} \norm{\un}_1 \\
& \leq C \norm{\muquer}_1 \cdot \eps_n \left(1 + C \frac{\eta_n}{\eps_n}\right) \leq C \norm{\muquer}_1 \cdot \eps_n.
\end{align*}
\end{proof}

Not surprisingly, the convergence rate only depends on the convergence of the data, since we assumed that the operator converges at least at the same rate (Assumption~\ref{ass_1}).

\begin{remark}
The estimate~\eqref{conv_rate} is consistent with existing theory for inverse problems with exact forward operators.
 If $B_n \equiv \exB$, the constraint $\exB u \leq \phi_n$ is essentially a bound on the (perhaps, weighted) $L^\infty$ norm of $Au - f$ for $f = \frac{f^u_n+f^l_n}{2}$.
 The case when the fidelity function is a characteristic function of the set $\{u \colon \norm{Au-f }\leq \delta\}$ was studied in~\cite[Thm 5.1]{Benning_Burger_general_fid_2011}, where the authors obtained the same convergence rate as~\eqref{conv_rate}.
\end{remark}

\section{Convergence of the level sets of $\un$}\label{level_sets}
Our goal in this section is to understand the structure of the minimisers $\un$ in the case of $\TV$-based regularisation.
 In particular, we want to know whether the level sets of $\un$ converge to those of $\uquerJ$, where $\uquerJ$ is the $\reg$-minimal solution of~\eqref{Au=f}, to which $\un$ converges.
In this section we consider $\reg(\cdot) = \TV(\cdot) + \gamma \norm{\cdot}_1$, where $\gamma$ is a small constant (recall that $\reg(\cdot) = \TV(\cdot)$ does not satisfy Assumption~\ref{ass_3}).
 We follow~\cite{Peyre:2017} and~\cite{iglesias_mercier_scherzer:2018}, where the authors proved Hausdorff convergence of the level sets of solutions of the ROF model~\cite{ROF} (for denoising in~\cite{Peyre:2017} and for general linear inverse problems in~\cite{iglesias_mercier_scherzer:2018}) to those of $\uquerJ$.
 In particular, if $\uquerJ$ is piecewise-constant, the authors of~\cite{Peyre:2017} conclude that the reconstructions are piecewise-constant outside the so-called extended support of the gradient of $\uquerJ$ in the low noise regime.

Our case requires several adjustments of the proofs in~\cite{Peyre:2017}.
 First,~\cite{Peyre:2017} considers $\reg(\cdot)=\TV(\cdot)$, while we need to consider $\reg(\cdot) = \TV(\cdot) + \gamma \norm{\cdot}_1$.
 Therefore, instead of considering sets satisfying $P(E) = \int_E p_n$ we need to consider sets satisfying $P(E) + \gamma \abs{E} = \int_E p_n$.
 Therefore, the level sets of $\un$ (as defined in~\cite{Peyre:2017}) solve the following optimisation problem (instead of the prescribed mean curvature problem):
\begin{equation}\label{Peyre:22}
\min_{X \subset \Omega} P(X) + \gamma\abs{X} - \int_X p_n.
\end{equation}
\noindent (Note that the case with the opposite sign of the integral does not occur since $\un \geq 0$).

Denote by $\Ent$ (or by $\En$, where this will cause no confusion) the level sets of $\un$.
 To prove that $\En$ indeed solves problem~\eqref{Peyre:22} we note that, since
\begin{equation*}
\int_\Omega gv \, dx = \int_\Omega \left( \int_0^\infty \one_{g(x) \geq t} v(x) \, dt \right) \, dx
\end{equation*}
\noindent for any functions $g,v$ such that $g \geq 0$, we get that 
\begin{eqnarray*}
&& (p_n,\un) = \int_0^\infty \left(\int_\Ent p_n\right) \,dt, \\
&& \norm{u_n}_1 = (u_n,\one) = \int_0^\infty \left(\int_\Ent \one\right) \,dt
\end{eqnarray*}
\noindent and $\TV(\un) = \int_0^\infty (\int_\Ent \TV(\one_\Ent)) \,dt$ by the coarea formula~\cite{Peyre:2017}.
 Combining this with $\reg(\un) = (p_n,\un)$, we get that
\begin{equation*}
\int_0^\infty \left( P(\Ent) + \gamma \abs{\Ent} - \int_\Ent p_n \right) \, dt = 0.
\end{equation*}
\noindent Since $p_n \in \dJ(0)$, the expression in the outer integral is non-negative and we get the desired equality
\begin{equation}\label{Per+Area=int_p}
P(\En) + \gamma \abs{\En} = \int_\En p_n.
\end{equation}
\noindent Since the objective in~\eqref{Peyre:22} is non-negative, $\En$ indeed solves~\eqref{Peyre:22}.

To prove Hausdorff convergence of the level sets of $\un$ to those of $\uquerJ$ along the lines of~\cite{Peyre:2017}, we need to prove Lemma 2 and Proposition 8 (following the notation of the arXiv version of the paper).
 The proofs in~\cite{Peyre:2017} rely on strong $L^2$ convergence of the subgradients, which we don't have in our case.
 However, weak-$^*$ convergence of $\Bnsmun$ in $L^\infty$ along with some orthogonality properties of $\lambda_n$ will be enough to obtain similar results, as we shall see.

Before we proceed with the proofs, let us note that for any level set $\Ent$, $t > 0$, the following inequality holds: $\one_\Ent \leq \frac{1}{t} \un$.
 Therefore, $0 \leq (\lambda_n,\one_\Ent) \leq \frac{1}{t} (\lambda_n,\un) = 0$.

\begin{lemma}(Lemma 2 in~\cite{Peyre:2017})
The level sets $\En$ have a finite perimeter and area.
\end{lemma}
\begin{proof}
Since $\En \subset \Omega$ and we assumed that $\Omega$ is bounded, finiteness of the area of $\En$ is trivial.
 For the perimeter $P(\En) = \TV(\one_\En)$ we obtain the following estimate using~\eqref{Per+Area=int_p}
\begin{equation*}
P(\En) \leq \reg(\one_\En) = (p_n,\one_\En) = (\lambda_n,\one_\En) - (\Bnsmun,\one_\En) = (-\Bnsmun,\one_\En) \leq \norm{\Bnsmun}_\infty \abs{\En} \leq C.
\end{equation*}
\end{proof}

\begin{proposition}(Proposition 8 in~\cite{Peyre:2017})
$\exists r_0 > 0$ such that $\forall r \in (0,r_0]$, $\forall \En$ and $\forall x \in \d \En$ the following estimates hold:
\begin{equation*}
\frac{\abs{B(x,r) \setplus \En}}{\abs{B(x,r)}} \geq C, \quad \frac{\abs{B(x,r)\setminus \En}}{\abs{B(x,r)}} \geq C, \quad C>0,
\end{equation*}
\noindent where $B(x,r)$ denotes a ball of radius $r$ centered at $x$.
\end{proposition}
\begin{proof}
Due to the optimality of $\En$ in problem~\eqref{Peyre:22}, we get that
\begin{equation*}
P(\En) + \gamma\abs{\En} - \int_\En p_n \leq P(\En \setminus B(x,r)) + \gamma \abs{\En \setminus B(x,r)} - \int_{\En \setminus B(x,r)} p_n,
\end{equation*}
\noindent which implies the following estimate:
\begin{equation*}
P(\En) - \int_\En p_n \leq P(\En \setminus B(x,r)) - \int_{\En \setminus B(x,r)} p_n - \gamma \abs{\En \cap B(x,r)}.
\end{equation*}

Geometric considerations yield:
\begin{equation*}
P(\En \cap B(x,r)) \leq \int_{\En \cap B(x,r)} p_n - \gamma \abs{\En \cap B(x,r)} + 2 \Hausdorff^1 (\d B(x,r) \cap \En).
\end{equation*}

For the first term on the left hand side we get the following estimate:
\begin{multline*}
\int_{\En \cap B(x,r)} p_n = (p_n,\one_{\En \cap B(x,r)}) = (\lambda_n,\one_{\En \cap B(x,r)}) - (\Bnsmun,\one_{\En \cap B(x,r)}) \\
\leq \norm{\Bnsmun}_\infty \cdot \norm{\one_{\En \cap B(x,r)}}_1 \leq \tilde C\abs{\En \cap B(x,r)}.
\end{multline*}

The isoperimetric inequality~\cite{Peyre:2017} yields:
\begin{equation*}
\sqrt{4\pi} \abs{\En \cap B(x,r)}^{1/2} \leq P(\En \cap B(x,r))
\end{equation*}

Denote $g(r) \defeq \abs{\En \cap B(x,r)}$.
 Then $g'(r) = \Hausdorff^1 (\d B(x,r) \cap \En)$ we get the following inequality:
\begin{equation*}
\sqrt{4\pi} \sqrt{g(r)} \leq 2 g'(r) + (\tilde C - \gamma) g(r).
\end{equation*}

Since $g(r) \to 0$ as $r \to 0$, for small $r$ we have that $g(r) \leq \sqrt{g(r)}$ and, therefore, 
\begin{equation*}
(\sqrt{4\pi} - \tilde C + \gamma) \leq \frac{d}{dr} \sqrt{g(r)}.
\end{equation*}

 
If the constants $C_0$ and $D_0$ from Assumption~\ref{ass_1} are small enough, the constant on the left hand side is positive and, integrating, we get that
\begin{equation*}
r(\sqrt{4\pi} - \tilde C + \gamma)  \leq \sqrt{g(r)}
\end{equation*}
\noindent and
\begin{equation*}
\frac{\abs{B(x,r) \setplus \En}}{\abs{B(x,r)}} \geq\frac{\left(\sqrt{4\pi} - \tilde C + \gamma \right)^2}{\pi}.
\end{equation*}

Comparing $\En$ with $\En \setplus B(x,r)$ in problem~\eqref{Peyre:22}, we get in a similar way the estimate
\begin{equation*}
\frac{\abs{B(x,r) \setminus \En}}{\abs{B(x,r)}} \geq\frac{\left(\sqrt{4\pi} - \tilde C -\gamma \right)^2}{\pi}.
\end{equation*}
\end{proof}

These results are sufficient to show Hausdorff convergence of the level sets of $\un$ to those of $\uquerJ$~\cite[Thm 1]{Peyre:2017},~\cite[Thm 2]{iglesias_mercier_scherzer:2018}.
 Similarly to Theorem~1 in~\cite{Peyre:2017}, one can also show that $\En \to E$ in the sense that $\lim_{n \to \infty} \abs{E \bigtriangleup \En} = 0$ and $P(E) + \gamma \abs{E} = \int_E \phat$.
 Indeed, passing to the limit in $\reg(\one_\En) = (p_n,\one_\En)$, we get that
\begin{multline*}
P(E) + \gamma \abs{E} \leq \liminf_{n \to \infty} P(\En) + \gamma \abs{\En} = \lim_{n \to \infty} (p_n,\one_\En) = \lim_{n \to \infty} (\lambda_n -\Bnsmun,\one_\En) \\
= \lim_{n \to \infty} (-\Bnsmun,\one_\En) = (-\exB^* \muhat,\one_E) = (\lhat-\exB^* \muhat,\one_E) = (\phat,\one_E)
\end{multline*}
\noindent due to $L^1$ convergence of $\one_\En \to \one_E$ and weak-$^*$ convergence of $\Bnsmun \wsto \exB^* \muhat$ in $L^\infty$.
 Therefore, $\TV(\one_E) < +\infty$ and $\one_E \in \BV$.
 Since $\phat \in \dJ_{\BV}(0)$, we get that $\reg(\one_E) \geq (\phat,\one_E)$ and, therefore, $(\phat,\one_E) = P(E) + \gamma \abs{E}$.

\begin{remark}
From $(\phat,\one_E) = \reg(\one_E)$ the authors of~\cite{Peyre:2017} conclude that $\d E$ is in the extended support of the gradient of $\uquerJ$.
 We can make a similar connection in the case $\frac{\eta_n}{\eps_n} \to 0$, when $\mun$ converges to a minimum-norm certificate (see Remark~\ref{min_norm_certificate}).
 However, due to non-uniqueness of the minimum norm certificate in our case, the definition of the extended support needs to be amended.
 We consider all $u$, whose subgradient contains a minimum-norm certificate $(\lmin,\mumin)$ solving~\eqref{eq:mumin}:
\begin{equation}
\Ext(\uquerJ) = \overline {\bigcup \{ \supp \abs{Du} \mid u \colon \exists (\lmin,\mumin) \text{ s.t. } \lmin - \exB^*\mumin \in \dJ(u) \}}.
\end{equation}
\end{remark}

\section{Debiasing and error estimation}\label{debiasing_and_error_est}
Two-step debiasing~\cite{Burger_Rasch_debiasing, Deladelle1, Deladelle2} aims at removing systematic bias in variational regularisation (such as loss of contrast with $\TV$) by solving an additional optimisation problem on the so-called model manifold defined as follows:
\begin{equation*}
\Mcal  = \{u \in L^1 \colon \reg(u) = (p_n,u)\},
\end{equation*}
\noindent where $p_n \in \dJ(u_n)$.
 The model manifold is the set of all elements of the solution space with zero Bregman distance to $\un$.
 In other words, it is the set of all elements sharing the subgradient with the approximate solution $\un$.
 Informally, the idea of two-step debiasing is that the approximate solution captures well the structure of the exact solution, such as the jump set in $\TV$-based regularisation, but is not perfect quantitatively due to a systematic bias introduced by the regulariser.
 This systematic bias is (partially) removed by optimising the fidelity term on the model manifold.

\subsection{Debiasing and model manifolds}\label{debiasing}

Our goal is to adapt the idea of debiasing to our specific setting.
 We assume that the first step, i.e. the solution of problem~\eqref{primal_n}, is already done and an approximate solution $\un$ is available along with the corresponding subgradient $p_n = \lambda_n - \Bnsmun$.

We slightly amend the definition of the model manifold for our specific setting. Fix some positive constants $\eps$ and $C$ and consider the following set:
\begin{equation}\label{Mn}
\Mn = \{u \geq 0 \colon B_n u \leq \phi_n, \,\, \reg(u) - (p_n,u) \leq \eps, \,\, (p_n,u-\un) \leq C  \}.
\end{equation}

We introduced two novel constraints as compared to the original feasible set in~\eqref{primal}.
 The inequality $\reg(u) - (p_n,u) \leq \eps$ is an upper bound\footnote{In the setting of~\cite{Burger_Rasch_debiasing} the Bregman distance is assumed to be zero, although the proposed numerical scheme allows some deviation.} on the Bregman distance between $\un$ and $u$.
 The condition $(p_n,u-\un) \leq C$ has a more technical nature and will be discussed in more detail in later (see Remark~\ref{drop_C}).

Next we examine some properties of the sets $\Mn$.
\begin{proposition}\label{uquer in Mn}
For sufficiently large $n$ any $\reg$-minimising solution $\uquerJ$ is an element of $\Mn$.
 
\end{proposition}
\begin{proof}
Since $(p_n,\uquerJ) \to \reg(\uquerJ)$ and $(p_n,\un) \to \reg(\uquerJ)$, we conclude that $(p_n,\uquerJ-\un) \leq C$ for sufficiently large $n$.
 Similarly, $\reg(\uquerJ) - (p_n,\uquerJ) = (\phat - p_n,\uquerJ) \to 0$ and therefore $\reg(\uquerJ) - (p_n,\uquerJ) \leq \eps$ for sufficiently large $n$.
\end{proof}

\begin{proposition}\label{v_n to uquer}
 $\reg(\cdot)$ is uniformly bounded on $\Mn$ and any sequence $v_n \in \Mn$ contains a subsequence (which we don't relabel) that strongly converges to a solution of~\eqref{Au=f} (not necessarily a $\reg$-minimising solution).
 \end{proposition}
 \begin{proof}
Indeed, $\forall u \in \Mn$ we have that
\begin{equation*}
\reg(u) \leq (p_n,u) + \eps \leq (p_n,\un) + C + \eps = \reg(\un) + C + \eps \leq \reg(\uquerJ) + C + \eps.
\end{equation*}
\noindent Since the sub-level sets of $\reg(\cdot)$ are sequentially compact, so are the sets $\Mn$ as closed subsets of a compact set.
 Therefore, we conclude that any sequence $v_n \in \Mn$ has a strongly convergent subsequence (that we don't relabel).
 Since only those $u$ that solve~\eqref{Au=f} belong to all sets $\Mn$ (for all $n$), we conclude that $v_n$ converges to a solution of~\eqref{Au=f}.
\end{proof}

\begin{remark}\label{drop_C}
Consider the expression $(p_n,u-\un)$ for some $u \geq 0$ such that $B_n u \leq \phi_n$.
 This expression is supposed to be bounded by a 'user-defined' constant $C$.
 Since $(p_n,\un) = \reg(\un) \leq \reg(\uquerJ)$, it is effectively an upper bound on $(p_n,u)$ on the feasible set.
 Consider the following estimate:
\begin{equation*}
(p_n,u) = (\lambda_n,u) - (\Bnsmun,u) \leq (\lambda_n,u) + \norm{\Bnsmun}_\infty \norm{u}_1.
\end{equation*}

If Assumption~\ref{ass_5} is satisfied, $\norm{u}_1$ is bounded and we get a bound $(p_n,u) \leq (\lambda_n,u) + C$, i.e., effectively, we only need an upper bound on $(\lambda_n,u)$ for all feasible $u$.
 Since $(\lambda_n,u_n) = 0$, we can drop the constant $C$ from the definition of $\Mn$ whenever $u_n>0$ a.e.


\end{remark}

Following~\cite{Burger_Rasch_debiasing}, to correct for the systematic bias of $\un$ we would need to optimise the data term on $\Mn$.
 However, since in our case the data term is the characteristic function of the set $\{u \geq 0 \colon B_n u \leq \phi_n\}$, optimising it on $\Mn$ does not make any sense (any element of $\Mn$ is a minimiser).
 A possible way around this would be to choose an operator $\tilde A_n \in [A^l_n,A^u_n]$ and a right-hand side $\tilde f_n \in [f^l_n, f^u_n]$, for example, $\tilde A_n = \frac{A^u_n + A^l_n}{2}$ and $\tilde f_n = \frac{f^u_n + f^l_n}{2}$, and optimise the discrepancy $\norm{\tilde A_n u - \tilde f_n}$ on $\Mn$.
 Convergence of the minimisers is guaranteed by Proposition~\ref{v_n to uquer} (since the sets $\Mn$ are closed, the minimisers also belong to $\Mn$ and, therefore, converge to a solution of~\eqref{Au=f}, possibly along a subsequence).

The choice of $\tilde A_n$ and $\tilde f_n$ depends on our additional assumptions about the nature of the errors in the operator and the data.
 For example, in the case of symmetric noise, the choice $\tilde A_n = \frac{A^u_n + A^l_n}{2}$ and $\tilde f_n = \frac{f^u_n + f^l_n}{2}$ is quite intuitive.

\subsection{Pointwise error estimates in constant regions}\label{error_bars}

Proposition~\ref{v_n to uquer} paves way for pointwise error estimates of $\TV$-regularised solutions in areas where the minimiser $\un$ is constant (by the results of Section~\ref{level_sets}, these areas converge to the areas where $\uquerJ$ is constant in the sense of Hausdorff convergence).
To obtain a meaningful result on the convergence of the pointwise bounds, we assume that the operator $A$ is injective and therefore the exact solution is unique (and will be denoted by $\uquer$).

We will make use of the following important property of the model manifold in case of $\TV$-based regularisation: as pointed out in~\cite{Burger_Rasch_debiasing}, the model manifold in the case $\reg(u) = \TV(u)$ contains all solutions that share the jump set with $\un$ (more precisely, they don't jump where $\un$ does not, but don't have to jump where $\un$ does). 
This is still valid in the case $\reg(u) = \TV(u) + \gamma \norm{u}_1$, as shown in Appendix~\ref{model_manifolds}.
 
\begin{theorem}[Pointwise error bars]
Suppose that
\begin{itemize}
\item $A$ is injective;
\item the exact solution $\uquer$ is piecewise-constant;
\item $\reg(u) = \TV(u) + \gamma \|u\|_1$, $\gamma = const >0$.
\end{itemize}
Denote any region where $\un$ is constant by $\omega_n$. Define $u^l_{\omega_n}$ and $u^u_{\omega_n}$ as follows
\begin{equation*}
u^l_n|_{\omega_n} = \frac{1}{\abs{\omega_n}}\min_{u \in \Mn} (u,\one_{\omega_n}), \quad u^u_n|_{\omega_n} = \frac{1}{\abs{\omega_n}} \max_{u \in \Mn} (u,\one_{\omega_n}).
\end{equation*}
Then $(u^u_{\omega_n} - u^l_{\omega_n}) \to 0$ and for sufficiently large $n$ we have that
\begin{equation*}
u^l_{\omega_n} \leq \uquer|_{\omega_n} \leq u^u_{\omega_n}.
\end{equation*}
\end{theorem}
\begin{proof}
Consider an arbitrary $u \in \Mn$ and denote its value $\omega_n$ by $\uun$. 
Consider the following linear functional:
\begin{equation}\label{func_u_on_omega}
(u,\one_\Un) = \uun \cdot \abs{\Un}.
\end{equation}

Since by Proposition~\ref{uquer in Mn} $\uquer \in \Mn$ if $n$ is sufficiently large, the jump sets of $\un$ and $\uquer$ coincide and we have that
\begin{equation*}
(\uquer,\one_\Un) = \uquer|_{\omega_n} \cdot \abs{\Un}.
\end{equation*}

Therefore, minimising and maximising this functional on $\Mn$ gives us a lower and an upper bound for $\uquer$.
By Proposition~\ref{v_n to uquer}, the $\argmin$ and $\argmax$ converge to a solution of~\eqref{Au=f}, which is unique by the injectivity of $A$. 
Therefore, by the continuity of the linear functional~\eqref{func_u_on_omega} we get that 
\begin{equation*}
u^u_{\omega_n} \defeq \max_{u \in \Mn} (u,\one_\Un) \to \uquer|_{\omega_n} \quad \text{and} \quad u^l_{\omega_n} \defeq \min_{u \in \Mn} (u,\one_\Un) \to \uquer|_{\omega_n},
\end{equation*}
\noindent proving the conjecture.
\end{proof}

\begin{remark}
Note that due to the fact that for a fixed $n$ $\uquer \in \Mn$ only for sufficiently small $\eps$, we can only guarantee that $\un$ captures the jump set of $\uquer$ in the limit. If we had the inclusion $\uquer \in \Mn$ with $\eps = 0$ for a fixed (but sufficiently large) $n$, the jump sets of $\un$ and $\uquer$ would coincide by the results of Appendix~\ref{model_manifolds}.
It is not clear, whether under any suitable assumptions $\uquer \in \Mn$ with $\eps = 0$ already for a fixed $n$, and can be an interesting direction of future research.
\end{remark}


\section{Numerical experiments}\label{numerical_experiments}

In this section we present numerical experiments illustrating the results of the previous sections.
 We concentrate on 1D examples in order to see the effects of different settings more clearly.
 We use CVX~\cite{cvx,cvx2} in all our experiments.

We consider deblurring with uncertainty in the blurring kernel, which has been studied in the partial-order based setting in~\cite{YK_JL:2018}.
Consider the signal shown in Fig.~\ref{pic-deblurring_signal} in blue (dashed line).
 This signal is convolved with a Gaussian blurring kernel with standard deviation $0.5$ and Dirichlet boundary conditions and then $2.5\%$ uniform noise is added to it.
 The blurred and noisy signal is shown in Fig.~\ref{pic-deblurring_signal} in green (solid line).
 Knowing the amount of noise in this signal, we can obtain lower and upper data bounds $f^l$ and $f^u$ as explained in~\cite{YK_JL:2018}. 

Being a convolution with a Gaussian kernel, the forward operator $A$ is injective and therefore the exact solution is unique.
 Assumption~\ref{ass_5} is also satisfied for a convolution operator, which implies that the $L^1$ norm of $u$ is bounded on the feasible set $\{u \geq 0 \colon B_n u \leq \phi_n\}$ and the regulariser $\reg(\cdot) = \TV(\cdot)$ satisfies the conditions of Theorem~\ref{thm:convergence}.
 However, Assumption~\ref{ass_3} is not satisfied for $\reg(\cdot) = \TV(\cdot)$ (see Appendix~\ref{App:dJ(0)}) and we cannot expect Hausdorff convergence of the level sets in this case (convergence rates~\eqref{conv_rate} do not apply either).

 In our  experiments we are going solve the following problem
 \begin{equation}\label{eq:primal_numerics}
 \min_{u \in L^1 \colon u \geq 0} \reg(u) \quad \text{s.t. } A^l u \leq f^u, \,\, A^u u \geq f^l
 \end{equation}
 \noindent with different choices of $A^l$ and $A^u$. We will use both $\reg(\cdot) = \TV(\cdot)$ and $\reg(\cdot) = \TV(\cdot) + \gamma \norm{u}_1$, where $\gamma$ is a small constant.

\paragraph{Reconstruction quality.}
Let us assume that only a slightly perturbed version $\tilde A$ of the blurring operator $A$ is available: 
\begin{equation*}
\tilde a_{ij} = \max\{a_{ij} + r_{ij} \cdot d,0\},
\end{equation*}
\noindent where $d=0.05*\max_{k,l} a_{kl}$ and $r_{ij}$ are i.i.d.~uniform random numbers with support $[-1,1]$ (i.e. the error in the operator is $5\%$). 
 Let us use the incorrect operator as if it were exact and solve~\eqref{eq:primal_numerics} with $A^l = A^u = \tilde A$ and $\reg(\cdot) = \TV(\cdot)$ (the results for $\reg(\cdot) = \TV(\cdot) + \gamma \norm{u}_1$ are similar). 
 As demonstrated in~\cite{YK_JL:2018}, this yields highly oscillatory solutions (Fig.~\ref{pic-reconstruction_wrong}).

 Knowing the amount of noise in the operator $\tilde A$, we can obtain lower and upper bounds $A^l$ and $A^u$ for the unknown exact operator $A$ (as also explained in~\cite{YK_JL:2018}).   
 Using these bounds, let us now reconstruct the signal by solving problem~\eqref{eq:primal_numerics} with $\reg(u) = \TV(u)$ and $\reg(u) = \TV(u) + \gamma \norm{u}_1$, where $\gamma = 10^{-4}$.
 The results are shown in Fig.~\ref{pic-reconstruction_TV} and~\ref{pic-reconstruction_TVL1}.

\begin{figure}[t]
    \centering
    \begin{subfigure}[t]{0.45\textwidth}
	\includegraphics[width=\textwidth]{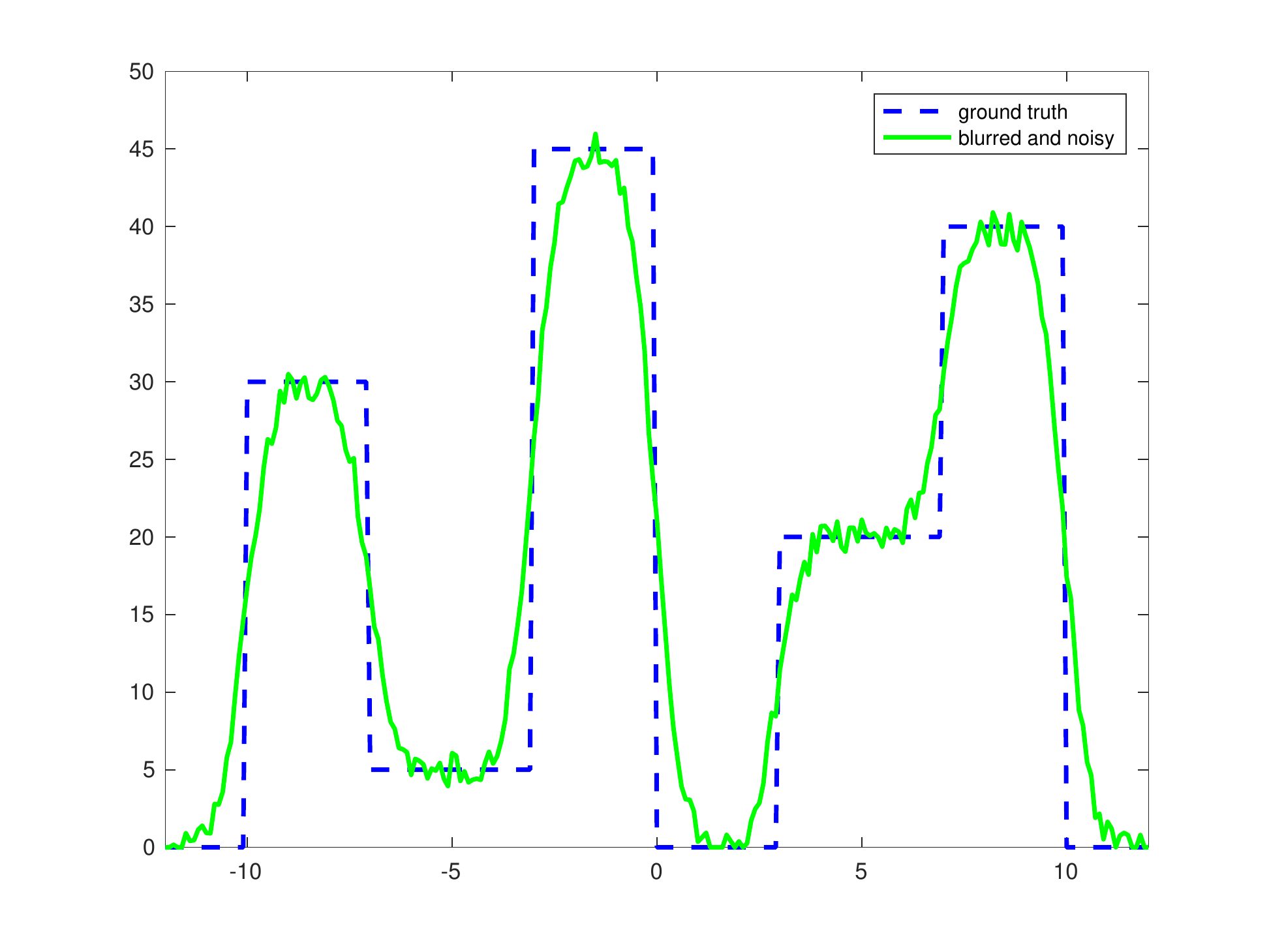}
	\subcaption{Ground truth (blue dashed line) and blurred and noisy signal (green solid line).
 $\PSNR = 18.3$, $\SSIM = 0.53$.}
	\label{pic-deblurring_signal}    
    \end{subfigure}
\begin{subfigure}[t]{0.45\textwidth}
            \centering
	\includegraphics[width=\textwidth]{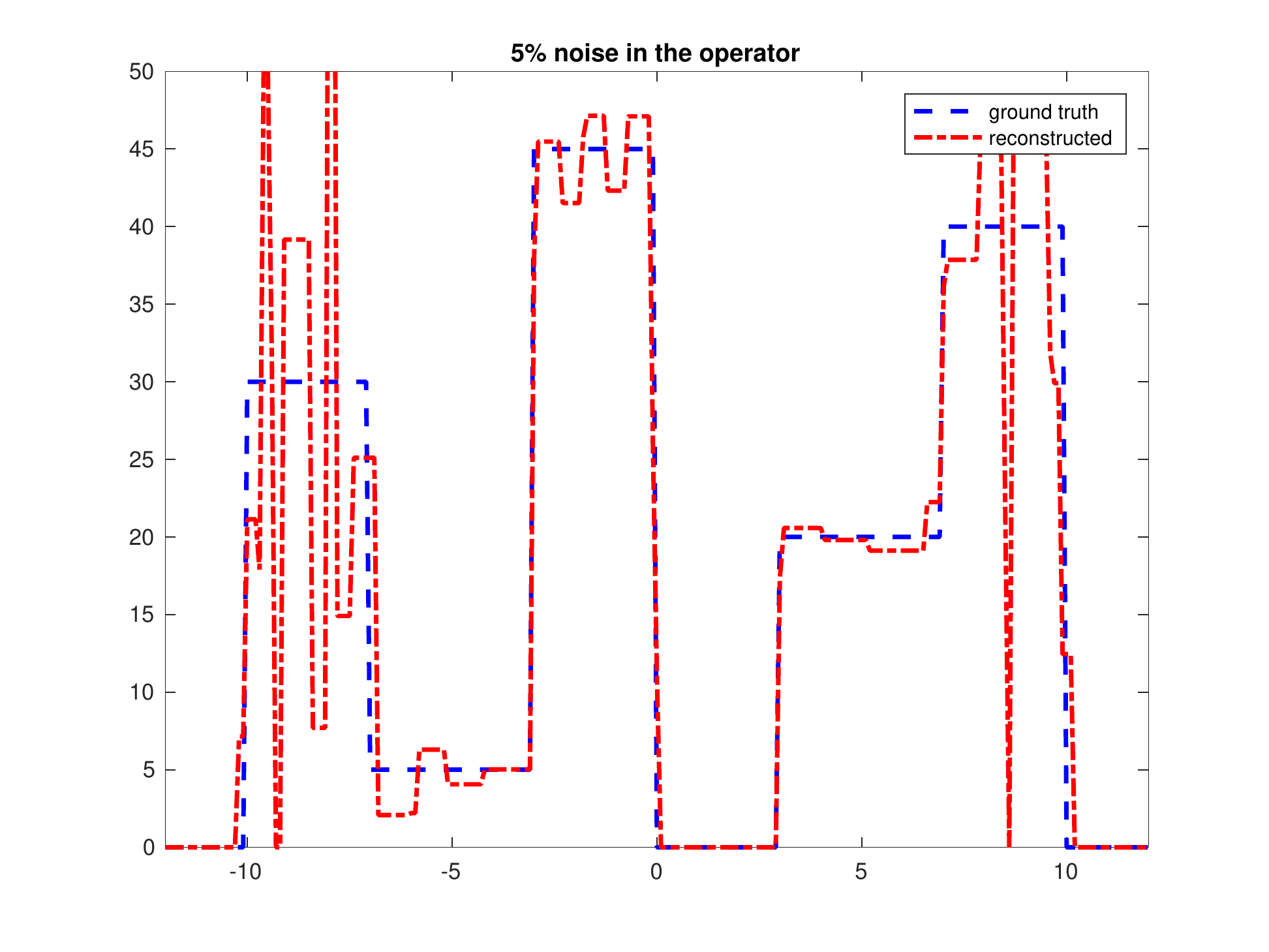}
	\subcaption{Reconstruction using a noisy operator (red dash-dotted line).
  $\PSNR = 15.3$, $\SSIM = 0.66$.
 }
	\label{pic-reconstruction_wrong}
    \end{subfigure}
    \\
\begin{subfigure}[t]{0.45\textwidth}
            \centering
	\includegraphics[width=\textwidth]{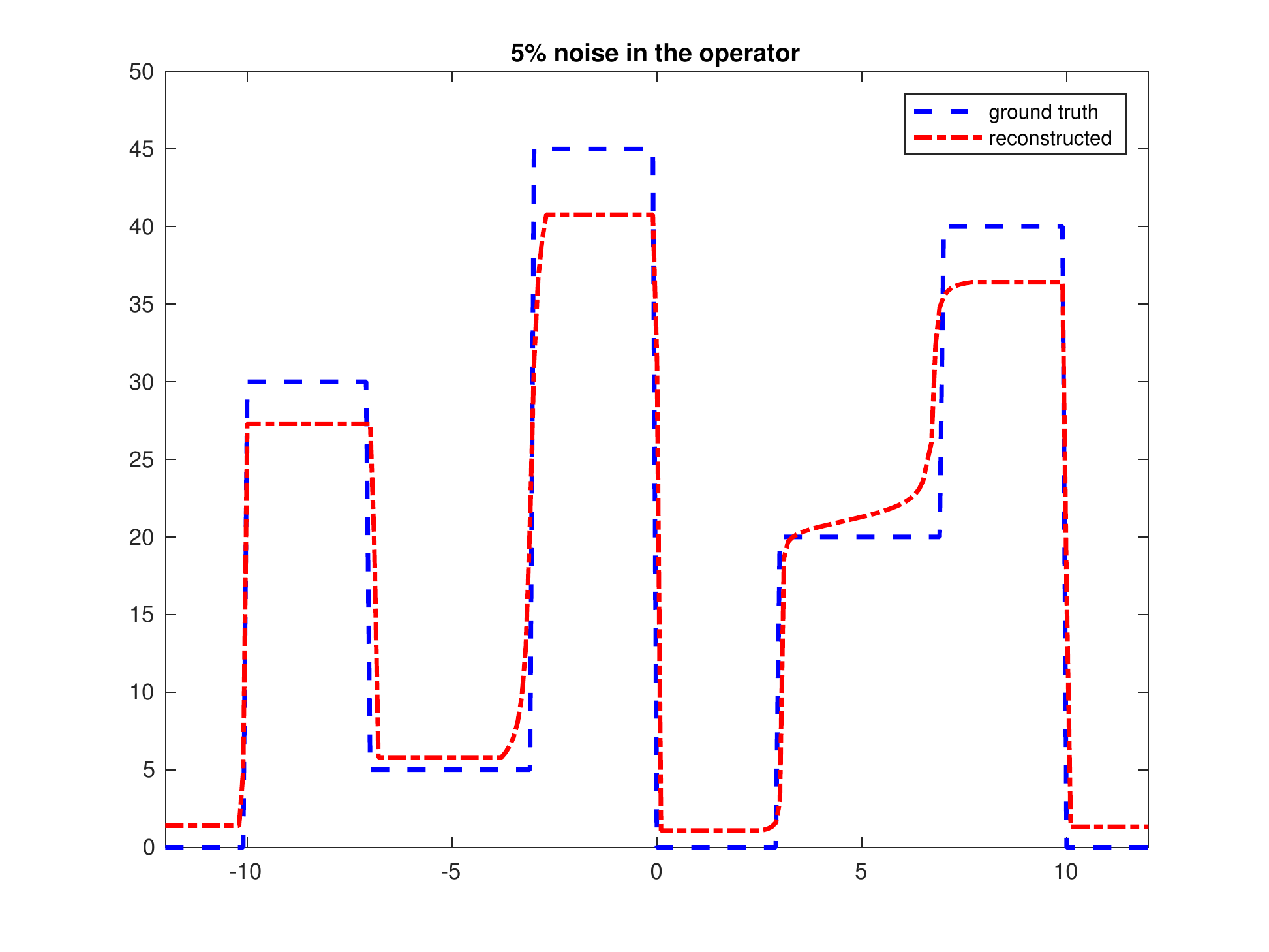}
	\subcaption{Interval-based reconstruction with $\reg(u) = \TV(u)$ (red dash-dotted line).
  \\ $\PSNR = 21.1$, $\SSIM = 0.67$.
 }
	\label{pic-reconstruction_TV}
    \end{subfigure}
\begin{subfigure}[t]{0.45\textwidth}
            \centering
	\includegraphics[width=\textwidth]{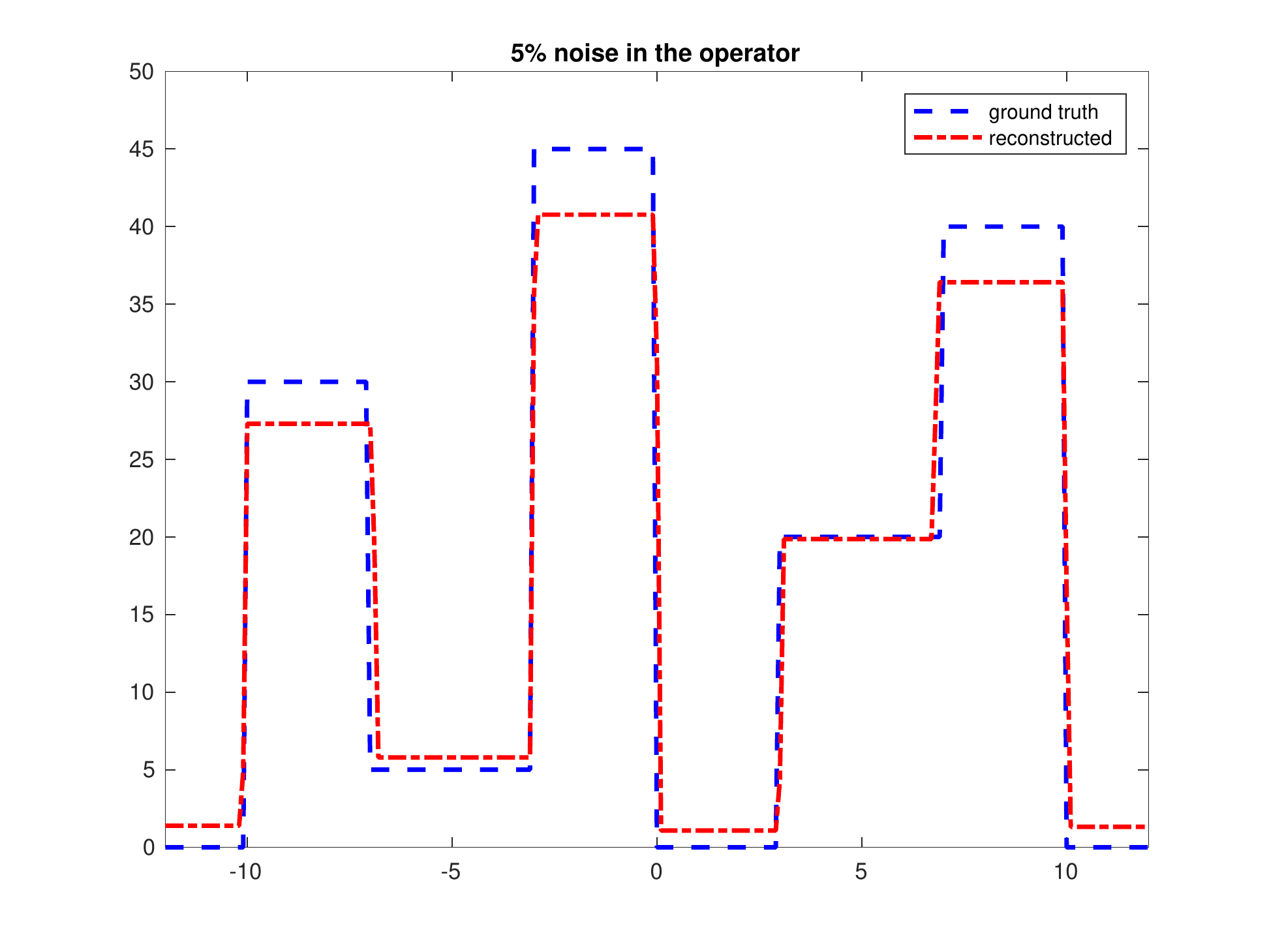}
	\subcaption{Interval-based reconstruction with $\reg(u) = \TV(u)+\gamma\norm{u}_1$, $\gamma = 10^{-4}$ (red dash-dotted line).
 $\PSNR = 21.9$, $\SSIM = 0.70$.}
	\label{pic-reconstruction_TVL1}
    \end{subfigure}
    \caption{Reconstruction using a noisy operator yields a highly oscillatory solution (\subref{pic-reconstruction_wrong}).
 The interval-based approach yields stable reconstructions (\subref{pic-reconstruction_TV} and \subref{pic-reconstruction_TVL1}).
 However, the geometric structure of the reconstructions is different. Reconstructions with $\reg(\cdot) = \TV(\cdot) + \gamma \norm{\cdot}_1$ are piecewise-constant (\subref{pic-reconstruction_TVL1}), whilst those with $\reg(\cdot) = \TV(\cdot)$ are not.}
\end{figure}

As expected, in both reconstructions the oscillations disappear and we obtain a stable reconstruction.
 But it is striking how much difference the small addition $\gamma \norm{\cdot}$ makes on the qualitative nature of the reconstruction.
 While the reconstruction with $\reg(\cdot) = \TV(\cdot) + \gamma \norm{u}_1$ has a structure very similar to that of the exact solution, the reconstruction based on plain $\TV(\cdot)$ is smooth in regions where the exact solution has jumps.
 We will discuss the structural properties of the reconstruction in both cases later on in this Section. 
 Let us note that for this signal the value of $\TV(\un)$ is about $180$, while $\gamma \norm{\un}_1$ is less than $0.5$.

It is worth noting that the value of $\TV$ of the reconstructions in Fig.~\ref{pic-reconstruction_TV} and~\ref{pic-reconstruction_TVL1} are identical up to machine precision.
 Therefore, the solution in Fig.~\ref{pic-reconstruction_TVL1} also solves problem~\eqref{eq:primal_numerics} with $\reg(\cdot) = \TV(\cdot)$ (the converse is not true, the $L^1$ norm of the solution in Fig.~\ref{pic-reconstruction_TV} is strictly greater than that of the solution in Fig.~\ref{pic-reconstruction_TVL1}).
 This demonstrates that the choice $\reg(\cdot) = \TV(\cdot)$ \emph{can} produce piecewise-constant reconstructions, while the choice $\reg(\cdot) = \TV(\cdot) + \gamma \norm{\cdot}_1$ produces them with a guarantee.


\paragraph{Structure of solutions.} The behaviour that $\TV$ demonstrates in Fig.~\ref{pic-reconstruction_TV} is surprising, since $\TV$ is known for introducing \emph{new} jumps, referred to as \emph{staircasing}~\cite{Ring:2000, Jalalzai:2016}, and not for overlooking existing ones.
 To better understand what happened in Fig.~\ref{pic-reconstruction_TV}, let us consider the simplest scenario, $B_n \equiv B \equiv E$, $\eps_n \leq u-f \leq \eps_n$, and solve the following problem:
\begin{equation*}
\min_{u \colon -\eps_n \leq u-f \leq \eps_n} \TV(u).
\end{equation*}

We follow the analysis in~\cite{Ring:2000}.
 In the one-dimensional case, the optimality condition reads as follows
\begin{equation*}
q' + \nu^u_n - \nu^l_n = 0,
\end{equation*}
\noindent where $q' \in \d\TV(u)$, $\abs{q} \leq 1$, and $\nu^u_n \neq 0$ when $u = f + \eps_n$, $\nu^l_n \neq 0$ when $u = f - \eps_n$.
 If neither of the bounds is active, we get that $q' = 0$ and $q$ can stay equal to $1$ until either the lower or the upper bound becomes active, resulting in piecewise-monotone reconstructions between the constant regions.

If we replace $\TV(\cdot)$ with $\TV(\cdot) + \gamma \norm{\cdot}_1$, $\gamma >0$, we get the following optimality condition (we assume that $f - \eps_n \geq 0$ and omit $\sign(u)$):
\begin{equation*}
\gamma + q' + \nu^u_n - \nu^l_n = 0.
\end{equation*}
\noindent Therefore, if neither of the bounds are active, we get that $q' = -\gamma < 0$ and $q$ cannot stay equal to $1$, resulting in piecewise-constant reconstructions.

In multiple dimensions the situation is different. The optimality condition in this case is as follows
\begin{equation*}
\div q + \nu^u_n - \nu^l_n = 0
\end{equation*}
\noindent for a smooth vector-field $q$ with $\norm{q}_\infty \leq 1$ and with inactive bounds we merely get that $\div q = 0$, which does not imply that $q$ is constant, in contrast to the one-dimensional case. 

An appropriate generalisation to multiple dimensions would be 
\begin{equation*}
\reg(u) = \norm{\div u}_1
\end{equation*}
\noindent for vector-valued images $u \in L^1(\Omega,\R^n)$~\cite{Briani:2012}. In this case one indeed has 
\begin{equation*}
\nabla q + \nu^u_n - \nu^l_n = 0
\end{equation*}
\noindent for a scalar field $q$ and $\nabla q = 0$ when neither of the bounds is active. Numerical experiments with vector-valued images are beyond the scope of this paper.

\paragraph{Comparison with Tikhonov-type regularisation.} For comparison, we solve the deblurring problem using a Tikhonov-type approach combined with Morozov's discrepancy principle (e.g.,~\cite{engl:1996}). Denote the exact signal by $f$, the noise level in the signal by $\delta$ and the noisy signal by $f_\delta$, so that we get that $\norm{f-f_\delta}_\infty \leq \delta$. We solve the following problem
\begin{equation}\label{eq:tikhonov_numerics}
 \min_{u \in L^1 \colon u \geq 0} \norm{\tilde A u - f_\delta}_\infty + \alpha \reg(u),
 \end{equation}
\noindent where $\tilde A$ is the noisy operator and $\alpha$ is chosen such that
\begin{equation} \label{eq:discr_pr_1}
 \norm{\tilde A u^\alpha_\delta - f_\delta}_\infty = C\delta,
 \end{equation}
 \noindent holds for the reconstructed signal $u^\alpha_\delta$ with a constant $C$ slightly greater than $1$ (we chose $C=1.01$). This approach is equivalent to~\eqref{eq:primal_numerics} with $A^l = A^u = \tilde A$, since in the absence of the operator error the constraints in~\eqref{eq:primal_numerics} are equivalent to a constraint on $\norm{\tilde A u - f_\delta}_\infty$ and $\alpha$ chosen according to~\eqref{eq:discr_pr_1} is just the Lagrange multiplier for this constraint. The result is shown in Fig.~\ref{pic-discr_pr_1}. Not surprisingly, we get the same kind of oscillations as in Fig.~\ref{pic-reconstruction_wrong}. 
 
 \begin{figure}[h!]
    \centering
    \begin{subfigure}[t]{0.45\textwidth}
	\includegraphics[width=\textwidth]{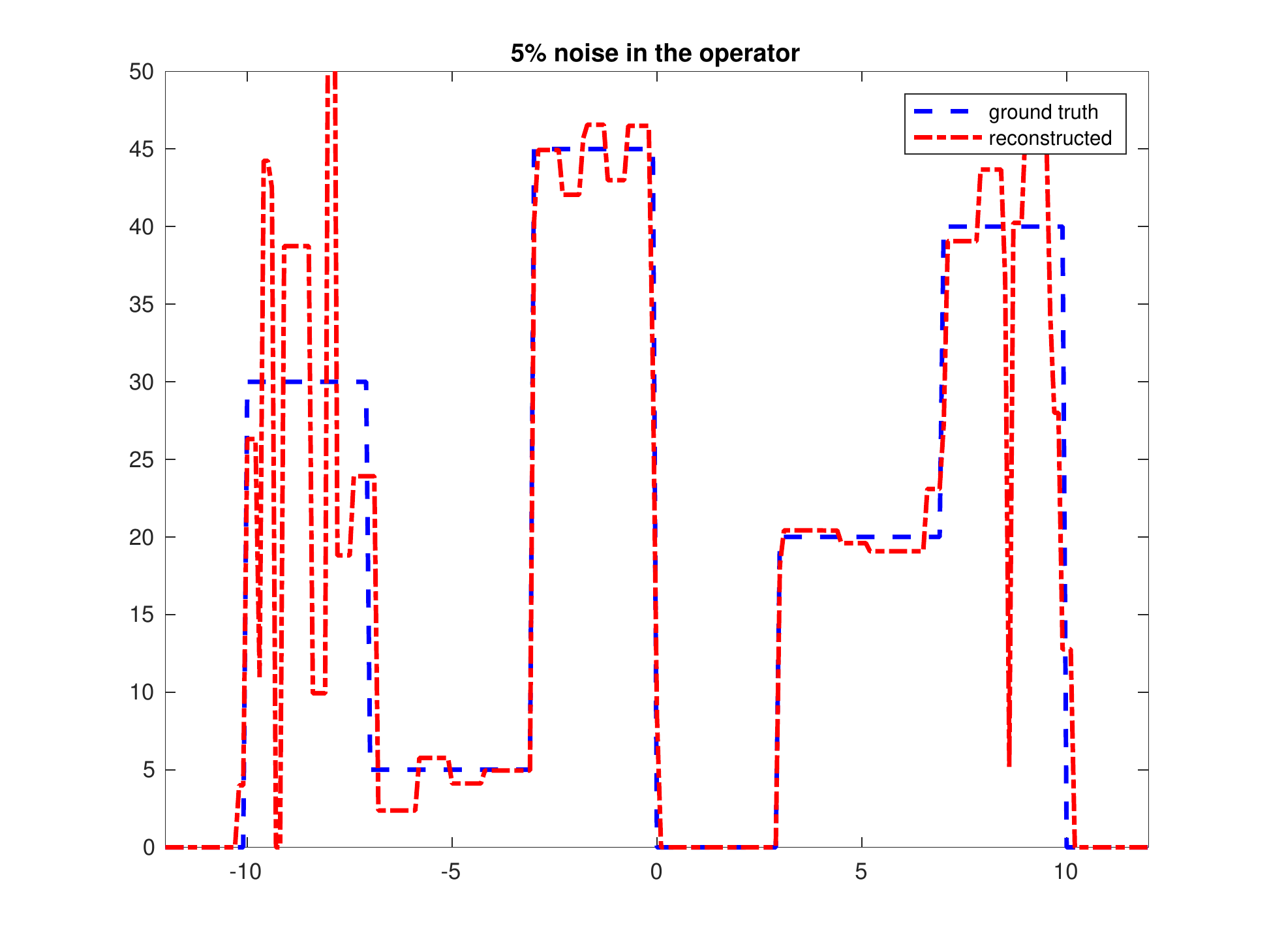}
	\subcaption{Tikhonov-type regularisation with $\alpha \colon \norm{\tilde A u^\alpha_\delta - f_\delta}_\infty = C\delta$. 
	\\ (red dash-dotted line). $\reg(u) = \TV(u)$. 
	\\ $5\%$ noise in the operator.
\\ $\PSNR = 16.4$, $\SSIM = 0.69$.}
	\label{pic-discr_pr_1}    
    \end{subfigure}
\begin{subfigure}[t]{0.45\textwidth}
            \centering
	\includegraphics[width=\textwidth]{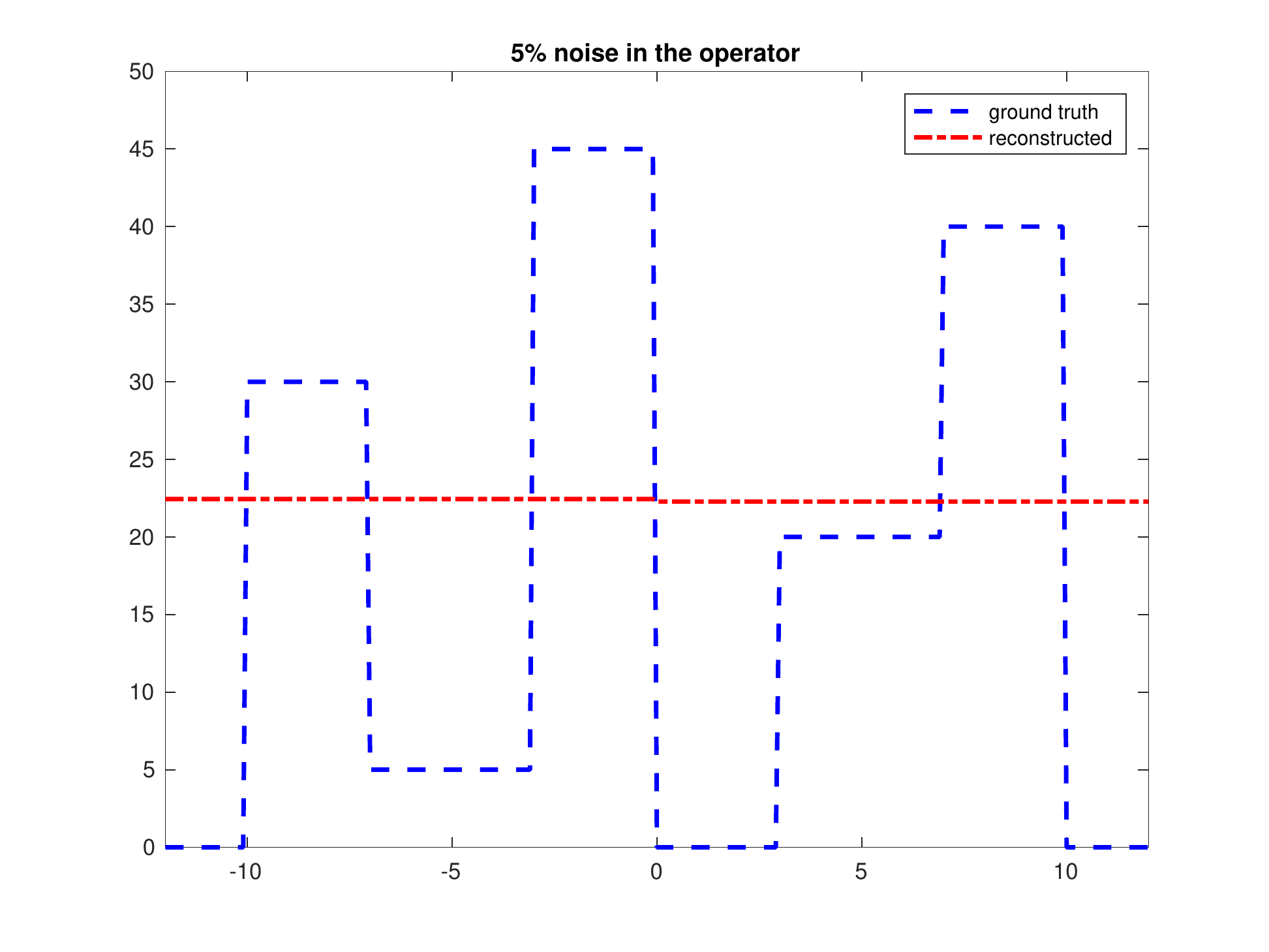}
	\subcaption{Tikhonov-type regularisation with $\alpha \colon \norm{\tilde A u^\alpha_\delta - f_\delta}_\infty = C (\delta + h \norm{u^\alpha_\delta}_1)$.
	\\ (red dash-dotted line). $\reg(u) = \TV(u)$. 
	\\ $5\%$ noise in the operator.
\\ $\PSNR = 9.10$, $\SSIM = 0.43$.
 }
	\label{pic-discr_pr_2_5percent}
    \end{subfigure}  
    \\
        \begin{subfigure}[t]{0.45\textwidth}
	\includegraphics[width=\textwidth]{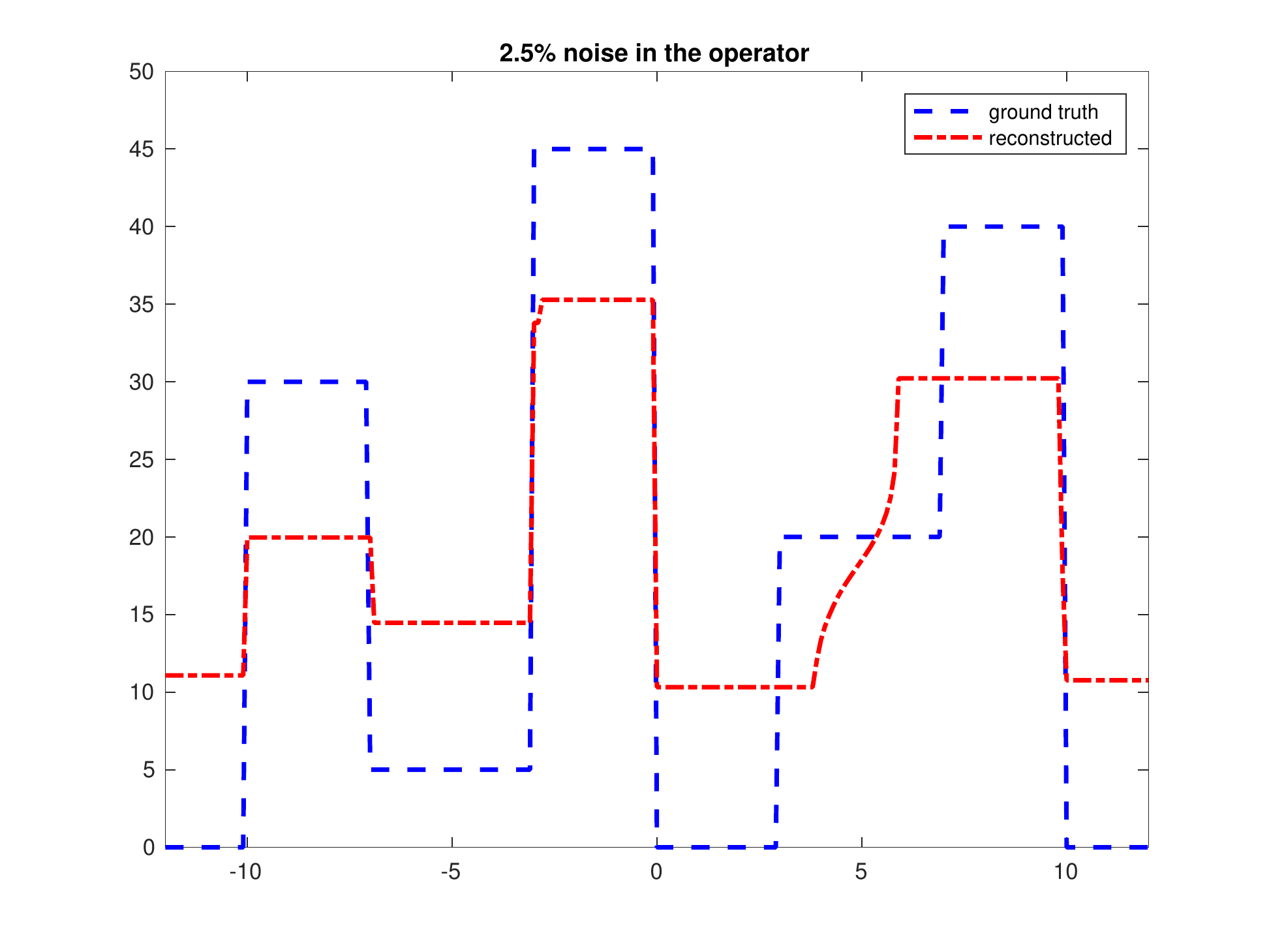}
	\subcaption{Tikhonov-type regularisation with $\alpha \colon \norm{\tilde A u^\alpha_\delta - f_\delta}_\infty = C (\delta + h \norm{u^\alpha_\delta}_1)$.
	\\ (red dash-dotted line). $\reg(u) = \TV(u)$.
	\\ $2.5\%$ noise in the operator.
\\ $\PSNR = 14.1$, $\SSIM = 0.51$.}
	\label{pic-discr_pr_2_25percent_TV}    
    \end{subfigure}
\begin{subfigure}[t]{0.45\textwidth}
            \centering
	\includegraphics[width=\textwidth]{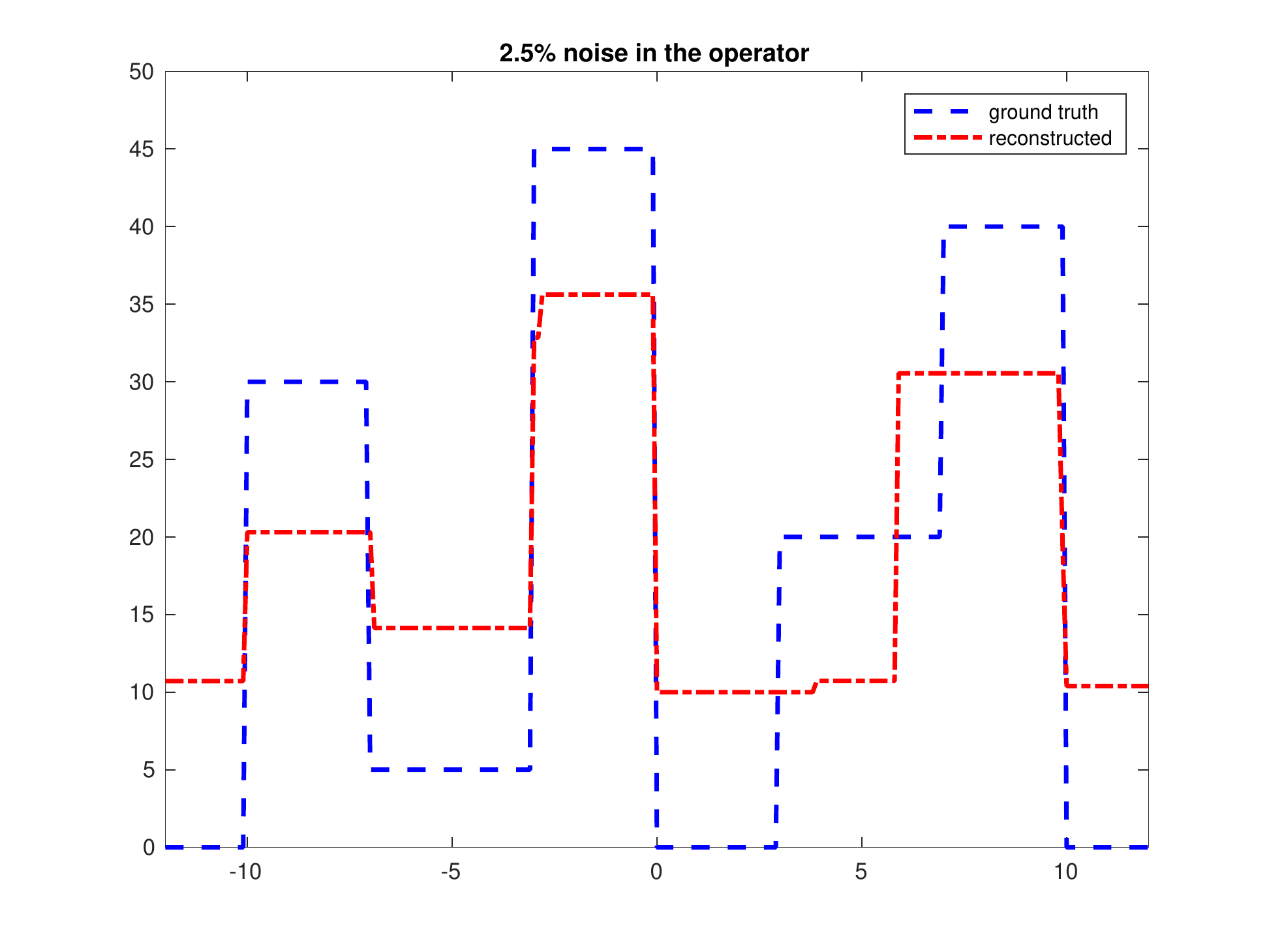}
	\subcaption{Tikhonov-type regularisation with $\alpha \colon \norm{\tilde A u^\alpha_\delta - f_\delta}_\infty = C (\delta + h \norm{u^\alpha_\delta}_1)$.
	\\ (red dash-dotted line). $\reg(u) = \TV(u)+\gamma\norm{u}_1$, $\gamma = 10^{-4}$. $2.5\%$ noise in the operator.
\\ $\PSNR = 14.1$, $\SSIM = 0.52$.
 }
	\label{pic-discr_pr_2_25percent_TVL1}
    \end{subfigure}    
    \caption{Tikhonov-type regularisation using a noisy operator combined with Morozov's discrepancy principle. Not accounting for operator errors results in oscillations~(\subref{pic-discr_pr_1}). Accounting for these errors within the framework of the discrepancy principle results in a severe loss of contrast. The solution is almost entirely flat for $5\%$ noise in the operator~(\subref{pic-discr_pr_2_5percent}). With $2.5\%$ noise in the operator the solution retains some similarity to the ground truth, but the loss of contrast is significant and the jump set is not correctly identified~(\subref{pic-discr_pr_2_25percent_TV}) and~(\subref{pic-discr_pr_2_25percent_TVL1}). The geometric structure of the reconstructions using $\reg(u) = \TV(u)$~(\subref{pic-discr_pr_2_25percent_TV}) and $\reg(u) = \TV(u)+\gamma\norm{u}_1$~(\subref{pic-discr_pr_2_25percent_TVL1}) is similar to that of the reconstructions obtained using our approach.}
\end{figure}

The reason for such oscillations is that the ground truth does not belong to the feasible set in~\eqref{eq:primal_numerics} with $A^l  = A^u = \tilde A$. A possible solution to this is modifying the feasible set so that the ground truth would become feasible. This could be achieved by replacing the constraint
\begin{equation}
\norm{\tilde A u - f_\delta}_\infty \leq \delta
\end{equation}
\noindent with
\begin{equation}\label{eq:fs_norm2}
\norm{\tilde A u - f_\delta}_\infty \leq \delta + h \norm{u}_1,
\end{equation}
\noindent where $h$ is the noise level in the operator, i.e. $h$ is such that $\norm{\tilde A - A}_{L^1 \to L^\infty} \leq h$. Using this constraint in the context of the residual method would result in a non-convex optimisation problem, however, in the context of Tikhonov-type regularisation it can be implemented in a convex manner using the following modification of the discrepancy principle (see~\cite{TGSYag} for the theory in Hilbert spaces)
\begin{equation} \label{eq:discr_pr_2}
 \norm{\tilde A u^\alpha_\delta - f_\delta}_\infty = C(\delta + h\norm{u}_1).
 \end{equation}

Since the constraint~\eqref{eq:fs_norm2} is rather conservative (it comes from the triangle inequality) and the feasible set is large, we could expect the regulariser to have a significant impact on the reconstruction. The results obtained using this approach are shown in Figs.~\ref{pic-discr_pr_2_5percent}--\ref{pic-discr_pr_2_25percent_TVL1}. With $5\%$ noise in the operator the regulariser ($\TV$ in this case) almost completely flattens out the reconstruction (Fig.~\ref{pic-discr_pr_2_5percent}). With less operator noise ($2.5\%$)  the reconstructions retain more structure, but we observe a significant loss of contrast (Figs.~\ref{pic-discr_pr_2_25percent_TV}--\ref{pic-discr_pr_2_25percent_TVL1}). We notice again the same difference in the structure of the solutions produced by $\reg(u) = \TV(u)$ (Fig.~\ref{pic-discr_pr_2_25percent_TV}) and $\reg(u) = \TV(u)+\gamma \norm{u}_1$ (Fig.~\ref{pic-discr_pr_2_25percent_TVL1}) as in Figs.~\ref{pic-reconstruction_TV}--\ref{pic-reconstruction_TVL1}.

The reason of the superior performance of the approach~\eqref{eq:primal_numerics} as compared to~\eqref{eq:tikhonov_numerics} with parameter choice~\eqref{eq:discr_pr_2} is that the feasible set in~\eqref{eq:primal_numerics} much smaller than that in~\eqref{eq:fs_norm2} (in fact, one can show that the feasible set in~\eqref{eq:primal_numerics} is a subset of~\eqref{eq:fs_norm2} if $\tilde A = \frac{A^u+A^l}{2}$ and $f_\delta = \frac{f^u+f^l}{2}$, see~\cite[Thm. 2]{Kor_IP:2014}).

\paragraph{Debiasing.} Although the reconstruction in Fig.~\ref{pic-reconstruction_TVL1} does well at capturing the qualitative structure of the solution, it still demonstrates a systematic bias in form of a loss of contrast.
 The same applies to the reconstruction in Fig.~\ref{pic-reconstruction_TV}.
 We will attempt to restore the contrast by optimising  on the set $\Mn$ (see~\eqref{Mn}) the discrepancy $\norm{\tilde A u - \tilde f}_2$, where $\tilde A$ and $\tilde f$ are the noisy operator and noisy data, respectively.

To define the the set $\Mn$, we need to fix two constants, $\eps$ and $C$.
 $\eps$ defines how close we want to stay to the model manifold; we choose to stay close and set $\eps = 10^{-6}$.
 Since in our case $\un > 0$ a.e., we can drop the constant $C$ from the definition of $\Mn$ (see Remark~\ref{drop_C}).

The results of debiasing applied to solutions in Fig.~\ref{pic-reconstruction_TV} and~\ref{pic-reconstruction_TVL1} are shown in Fig.~\ref{pic-debiasing_TV} and~\ref{pic-debiasing_TVL1}, respectively.
 We see that debiasing was able to almost perfectly recover the ground truth in both cases, although the qualitative nature of the reconstruction with $\reg(\cdot) = \TV(\cdot)$ is quite different from that of the ground truth.
 Note also that naive reconstruction with the noisy operator produced oscillatory results shown in Fig.~\ref{pic-reconstruction_wrong}, whilst the two-step approach involving solving problem~\eqref{primal} and debiasing nearly perfectly recovered the ground truth.

\begin{figure}[t]
    \centering
\begin{subfigure}[t]{0.45\textwidth}
            \centering
	\includegraphics[width=\textwidth]{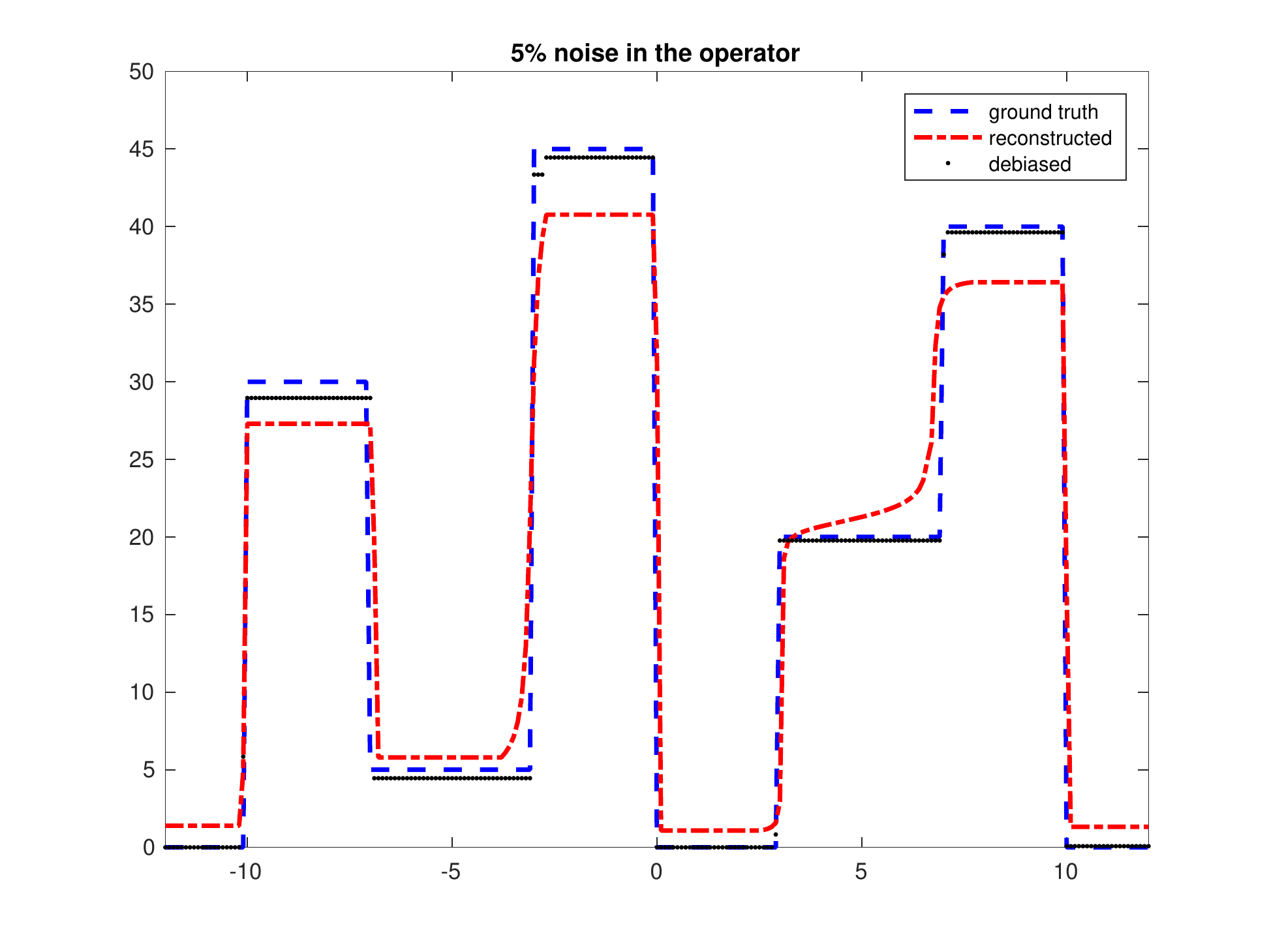}
	\subcaption{Debiased solution with $\reg(u) = \TV(u)$ (black dotted line).
  \\ $\PSNR = 29.5$, $\SSIM = 0.97$.
 }
	\label{pic-debiasing_TV}
    \end{subfigure}
\begin{subfigure}[t]{0.45\textwidth}
            \centering
	\includegraphics[width=\textwidth]{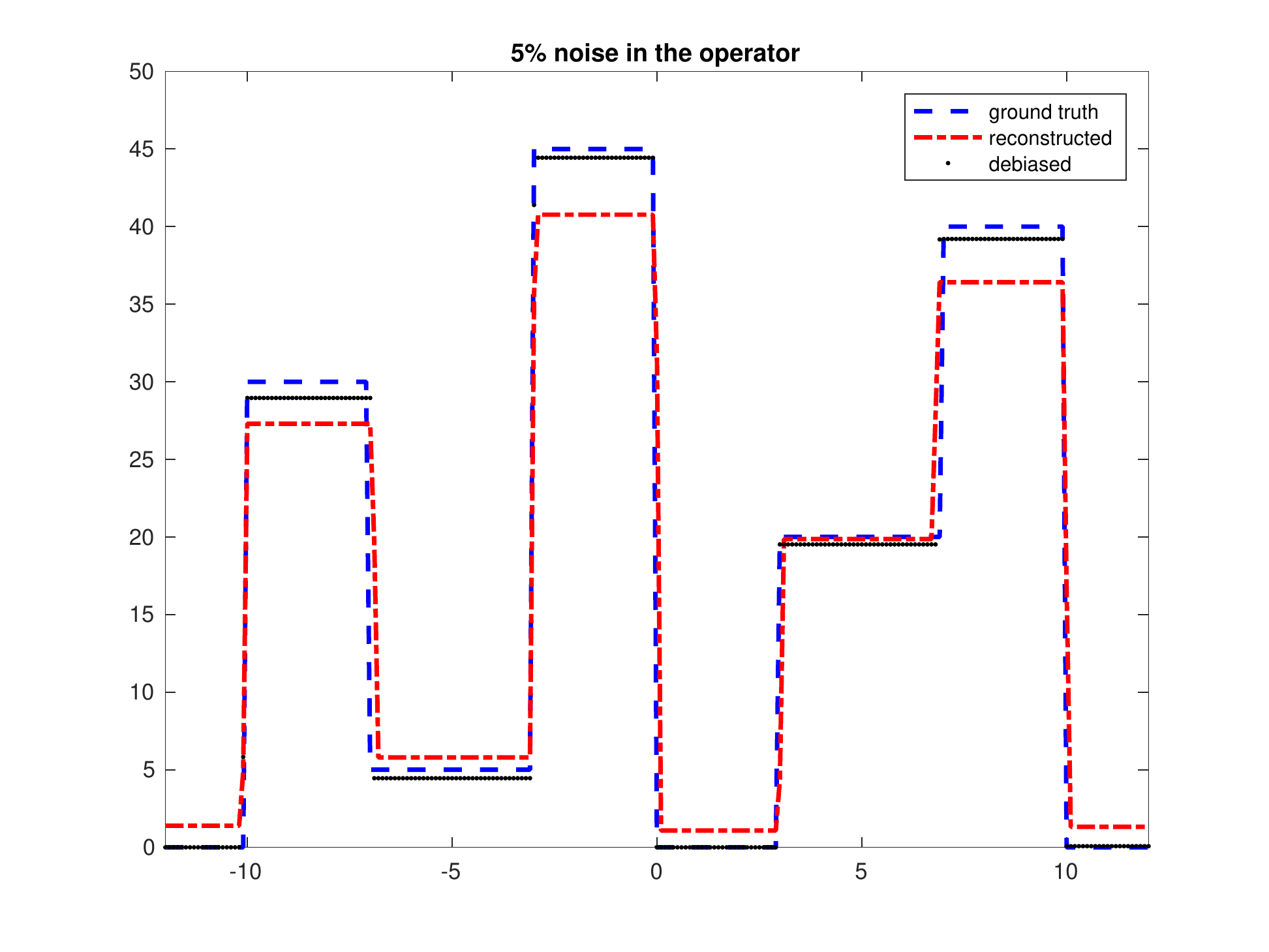}
	\subcaption{Debiased solution with $\reg(u) = \TV(u)+10^{-4}\norm{u}_1$ (black dotted line).
\\ $\PSNR = 27.5$, $\SSIM = 0.95$.}
	\label{pic-debiasing_TVL1}
    \end{subfigure}
    \caption{Debiasing almost perfectly recovers the exact solution in both cases, although the reconstruction using $\reg(u) = \TV(u)$ is quite different qualitatively from the ground truth.}
\end{figure}


\paragraph{Error bars.} The results of Section~\ref{error_bars} allow us to provide pointwise error estimates in regions where the minimiser $\un$ is constant.
 Therefore, we need to guarantee that $\un$ is piecewise-constant if the exact solution $\uquer$ is.
 We can only guarantee this for the case $\reg(u) = \TV(u)+\gamma \norm{u}_1$, therefore, we will only consider this case.

In order to provide a pointwise error estimate for a piecewise constant solution $\un$, we need to automatically determine the regions where it is constant.
 We proceed as follows.
 First observe that since $\norm{\cdot}_1$ is continuous at $0 \in \dom(\TV)$, we have that $\dJ(u) = \d\TV(u)+\gamma \d\norm{u}_1$~\cite{Borwein_Zhu}.
 Any $p \in \d\TV(u)$ can be written as a divergence of some function $q \in L^\infty$, $\norm{q}_\infty \leq 1$, such that $(\nabla \cdot q, u) = \TV(u)$~\cite{Burger_Osher_TV_Zoo}.
 The latter equality can be rewritten as  $(-q, \nabla u) = \TV(u)$, since the gradient is the adjoint of the negative divergence.
 Taking into account that $\d\norm{u}_1 = \{y \in L^\infty \colon \norm{y}_\infty \leq 1, \, \norm{u}_1 = (y,u)\}$, we get the following expression:
\begin{equation}
p_n = y_n + \nabla \cdot q_n, \quad \norm{y_n}_\infty \leq 1, \, \norm{\un}_1 = (y_n,\un), \,\,\norm{q_n}_\infty \leq 1, \, (-q_n, \nabla \un) = \TV(\un).
\end{equation}

This function $q_n$ contains the information about jumps of $\un$: whenever $\abs{q_n} < 1$, $\un$ has to be constant~\cite{Burger_Rasch_debiasing}.
 Therefore, we can locate jumps of $\un$ by finding points where $q_n = \pm 1$.
 In general, $q_n$ (as well as $y_n$) will be non-unique, but we can pick one solving the following optimisation problems:
\begin{equation}\label{y_n}
\min_{y\colon \norm{y}_\infty \leq 1} \norm{y}_1 \quad \text{s.t. } \norm{\un}_1 = (y,\un)
\end{equation}
\noindent and
\begin{equation}\label{q_n1}
\min_{q\colon \norm{q}_\infty \leq 1} \norm{q}_1 \quad \text{s.t. } p_n = \nabla \cdot q + \gamma y_n.
\end{equation}

Finding where $\abs{q_n} > 1 - \nu$ for some small constant  $\nu$  (we took $\nu = 10^{-6}$ in our experiments), we can locate the jumps of $\un$.
 Alternatively, instead of solving~\eqref{y_n} and~\eqref{q_n1}, we can solve the following problem:
\begin{equation}\label{q_n2}
\min_{q\colon \norm{q}_\infty \leq 1} \norm{q}_1 \quad \text{s.t. } \TV(\un) = (-q, \nabla \un).
\end{equation}

Both methods gave the same results in our experiments, although the method based on~\eqref{q_n2} was much less sensitive to the cut-off constant $\nu$.
 

Having identified regions where $\un$ is constant, we can proceed with finding pointwise error bounds as described in Section~\ref{error_bars}. 
 We present results for 
 $10\%$ noise in the operator (Fig.~\ref{pic-error_bars_10}) and $5\%$ noise (Fig.~\ref{pic-error_bars_5}). 
  First of all, we see that the exact solution $\uquer$ is indeed contained within the bounds, together with the approximate solution $\un$ and the debiased solution.
 As expected, the error bars get tighter as the error in the operator gets smaller.
One can also notice that the minimiser $\un$ often lies 'on the boundary' of the feasible set, its values coinciding with either the lower of the upper bound in almost all intervals. 
 We also see that with $10\%$ operator noise the reconstruction $\un$ has a small additional jump at the value of the argument of around $8$ that also becomes clearly visible in the lower bound (see Fig.~\ref{pic-error_bars_10}).

\begin{figure}[t]
    \centering
\begin{subfigure}[t]{0.45\textwidth}
            \centering
	\includegraphics[width=\textwidth]{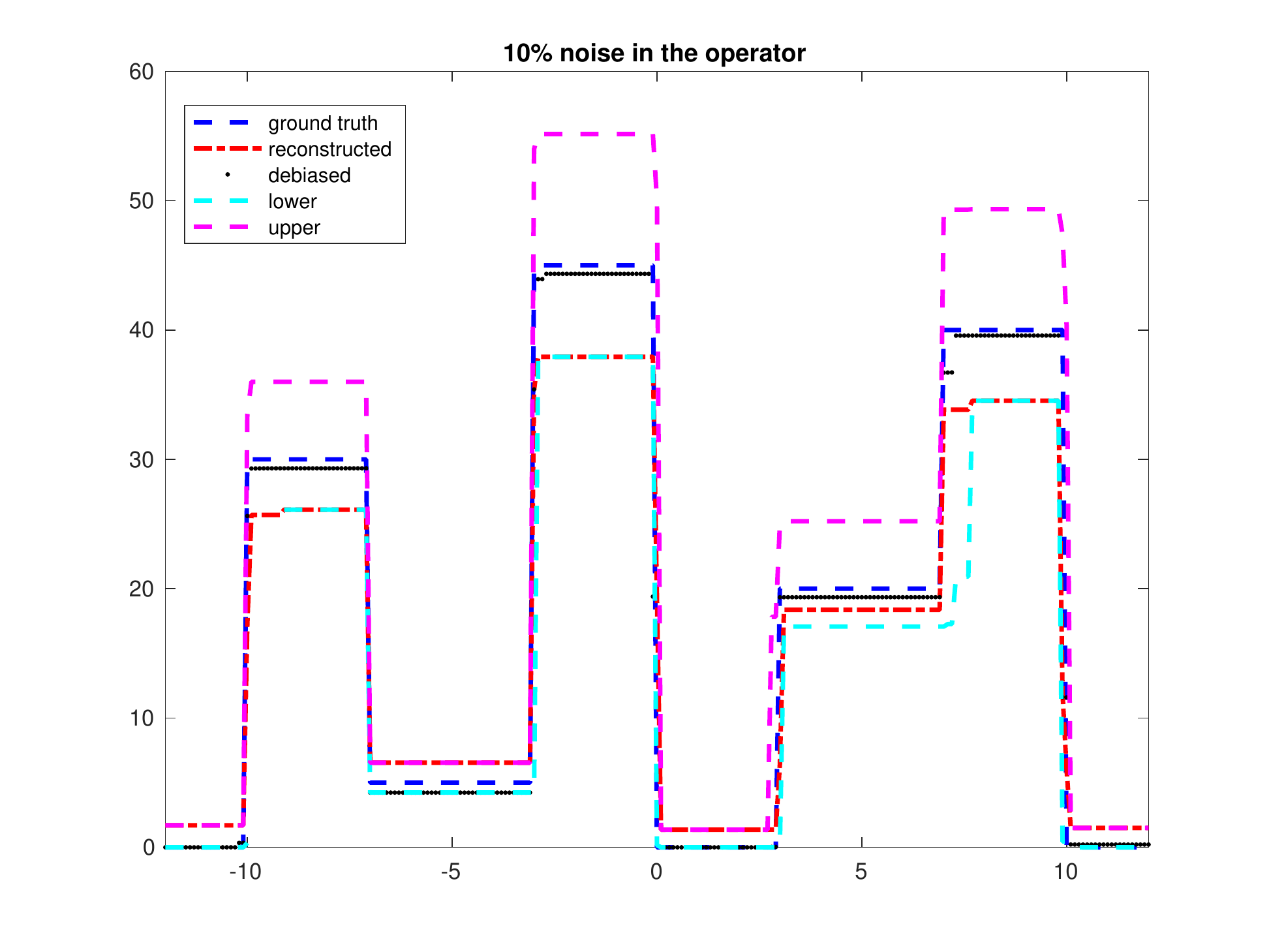}
	\subcaption{Pointwise error bounds with $\reg(u) = \TV(u)+10^{-4}\norm{u}_1$ (cyan and magenta dashed lines).
 $10\%$ noise in the operator.}
	\label{pic-error_bars_10}
    \end{subfigure}
\begin{subfigure}[t]{0.45\textwidth}
            \centering
	\includegraphics[width=\textwidth]{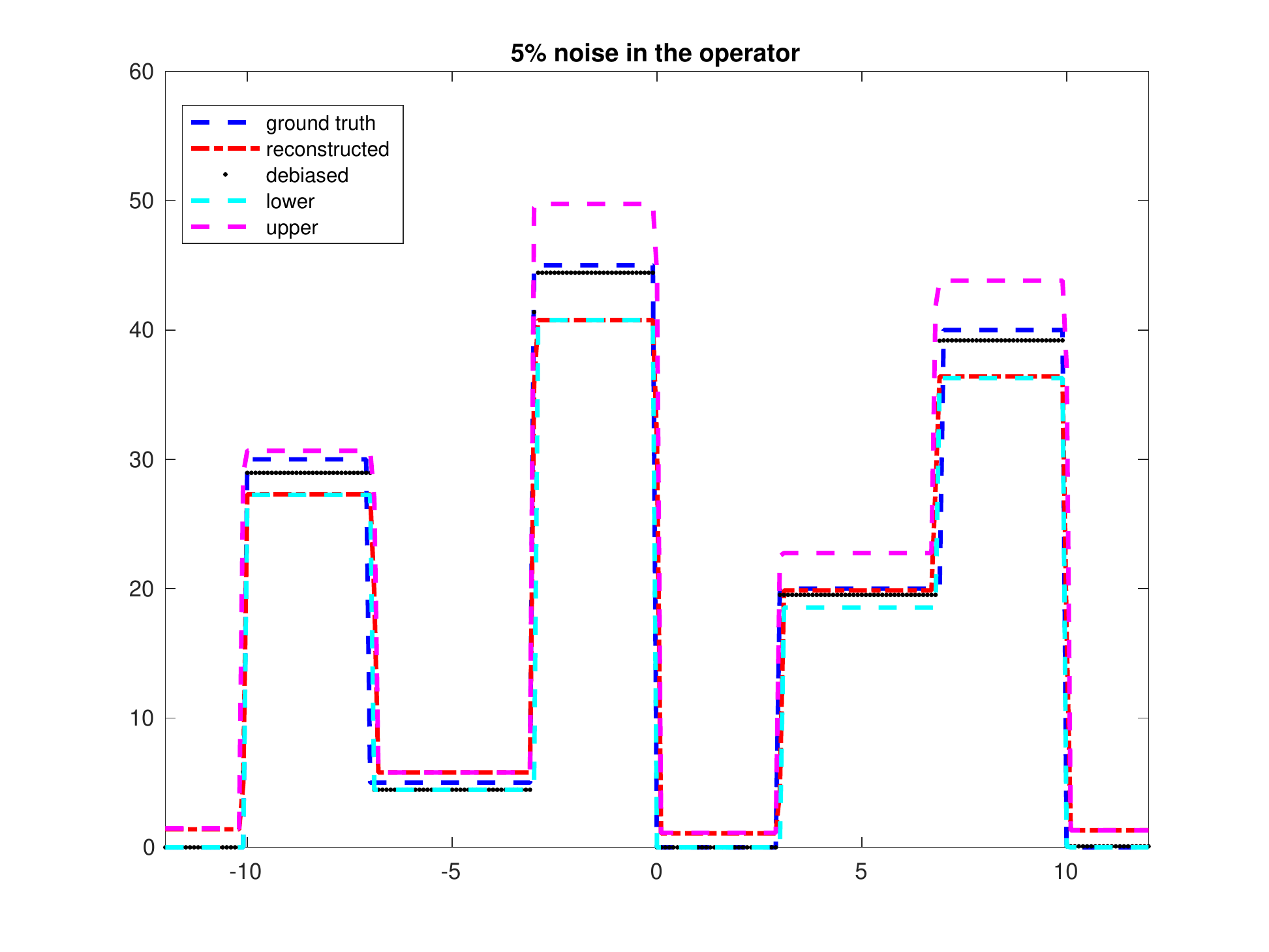}
	\subcaption{Pointwise error bounds with $\reg(u) = \TV(u)+10^{-4}\norm{u}_1$ (cyan and magenta dashed lines).
 $5\%$ noise in the operator.}
	\label{pic-error_bars_5}
    \end{subfigure}
    \caption{The error bounds contain the exact solution $\uquer$, the approximate solution $\un$ and the debiased solution.
 The error bars get closer as the error in the operator gets smaller.
 When the minimiser $\un$ correctly detects the jump set of the exact solution, the error bars follow the structure of the exact solution (\subref{pic-error_bars_5}), otherwise they may contain additional jumps (\subref{pic-error_bars_10}).}
\end{figure}

\section{Conclusions}
The paper presents a theoretical analysis of inverse problems with imperfect forward models in the setting of variational regularisation using one-homogeneous functionals.
 Convergence rates in terms of Bregman distances are obtained that coincide with existing results on inverse problems with exact operators, providing a natural generalisation of the existing theory.

An important aspect of the paper is the study of the interplay between the errors in the data and the operator.
 It turned out that errors in the data should not converge faster than the errors in the operator for the theory to work.
 This result is rather intuitive: there is no need to measure something more precisely than we can predict it.
 Therefore, it might be useful in practice to artificially decrease the quality of the data in order to match that of the operator.
 Along the same lines goes the observation that the data should converge in such a way that there is always a uniform gap between the upper and the lower bound, cf.
 Assumption~\ref{ass_1}.

In the special case of $\TV$-based regularisation we obtained results on the convergence of the level sets of the approximate solutions to those of the ground truth, building on recent work by Chambolle et.al. on problems with exact operators.
 It turned out that, unlike the classical case with an $L^2$ fidelity, $\TV$ does not guarantee the convergence of the level sets, while $\TV + \gamma \norm{\cdot}_1$, $\gamma>0$, does.
 The deciding property in this respect is having a subdifferential at $0$ with non-empty interior, which holds for $\TV + \gamma \norm{\cdot}_1$, $\gamma>0$, but fails for plain $\TV$.

Using our theoretical results we generalised the concept of two-step debiasing to problems with imperfect operators and proposed a method of obtaining asymptotic pointwise lower and upper bounds for the ground truth if it is piecewise constant, demonstrating the performance of both techniques in numerical experiments.

\section*{Acknowledgments}
MB acknowledges the support of ERC via Grant EU FP 7 -- ERC Consolidator Grant 615216 LifeInverse. 
A significant portion of the work presented in this paper was done while YK was a Humboldt Fellow at the University of M\"unster.
 YK acknowledges the support of the Humboldt Foundation in that period.
 Currently YK holds a Newton International Fellowship sponsored by the Royal Society, whose support he also acknowledges.

\appendix
\section{Absolutely one-homogeneous regularisation functionals}\label{app:abs_one_homogeneous}
A functional $\reg(\cdot)$ is called absolutely one-homogeneous if
\begin{equation*}
\reg(su) = \abs{s} \reg(u) \quad \forall s \in \R, \,\, \forall u \in L^1,
\end{equation*}
\noindent Absolutely one-homogeneous functionals are widely used in regularisation and play a crucial role, for instance, in non-linear spectral theory~\cite{Burger_Gilboa_Moeller_Eckard_spectral}.

Absolutely one-homogeneous convex functionals have some useful properties, for example, it is obvious that $\reg(0)=0$.
 Some further properties are listed below.
\begin{proposition}\label{prop:one_homog}
Let $\reg(\cdot)$ be a convex absolutely one-homogeneous functional and let $p \in \dJ(u)$.
 Then the following equality holds:
\begin{equation*}
\reg(u) = (p,u).
\end{equation*}
\end{proposition}
\begin{proof}
Indeed, consider the (generalised) Bregman distance~\cite{Bregman}
\begin{equation*}
D_\reg^p(v,u) = \reg(v) - \reg(u) - (p,v-u) \geq 0 \,\, \forall v.
\end{equation*}

Taking $v=0$, we get that $\reg(u) \leq (p,u)$, while taking $v=2u$ and noting that $\reg(v) = 2\reg(u)$, we get that $\reg(u) \geq (p,u)$, hence $\reg(u) = (p,u)$.
\end{proof}
\begin{remark} The Bregman distance $D_\reg^p(v,u)$ in this case can be written as follows:
\begin{equation*}
D_\reg^p(v,u) = \reg(v)  - (p,v).
\end{equation*}
\end{remark}

\begin{proposition}\label{prop:dJ(0)}
Let $\reg(\cdot)$ be a convex absolutely one-homogeneous functional.
 Then the convex conjugate $\reg^*(\cdot)$ is the characteristic function of the convex set $\dJ(0)$.
\end{proposition}
\begin{proof}
By the definition of the convex conjugate, we have that
\begin{equation*}
\reg^*(p) = \sup_u ((p,u) - \reg(u)).
\end{equation*}
\noindent Since $0$ is a feasible element, the supremum is $\geq 0$ for all $p$.
 If $(p,u_0) - \reg(u_0) >0$ for some $u_0$, then, choosing $u=C\cdot u_0$ with an arbitrary $C>0$, we get that the supremum is unbounded and $\reg^*(p) = +\infty$.
 Therefore, for all $p$ s.t. $\reg^*(p) < +\infty$ we have that $(p,u) - \reg(u) \leq 0$ $\forall u$.
  By the definition of a subgradient we get that $p \in \dJ(0)$.
 Since $(p,u) - \reg(u) \leq 0$ $\forall u$ and $\sup_u ((p,u) - \reg(u)) \geq 0$, we conclude that $\reg^*(p) = 0$ whenever $\reg^*(p) < +\infty$, hence the assertion.
\end{proof}

An obvious consequence of the above results is the following
\begin{proposition}\label{char_of_dJ(u)}
For any $u \in \U$, $p \in \dJ(u)$ if and only if $p \in \dJ(0)$ and $\reg(u) = (p,u)$.
\end{proposition}

\section{Properties of $\subdiff \reg(0)$}\label{App:dJ(0)}
In this section we discuss two different classes of regularistaion functionals for which Assumption~\ref{ass_3} is satisfied (or not).
 First let us start with functionals of the form
\begin{equation*}
\reg(u) = \norm{u}_1 + g(u),
\end{equation*}
\noindent where $g \geq 0$ is an absolutely one-homogeneous functional.
 Since $\reg(u) \geq \norm{u}_1$ and convex conjugation is order reversing~\cite{Borwein_Zhu}, we get that
\begin{equation}
\charf_{\dJ(0)} (\cdot) = \dJ^*(\cdot) \leq (\norm{\cdot}_1)^* = \charf_{\norm{\cdot}_\infty \leq 1},
\end{equation}
\noindent where $\charf_X (\cdot)$ is the characteristic function of the set $X$.
 Therefore,  the inclusion $\{p \colon \norm{p}_\infty \leq 1\} \subset \dJ(0)$ holds and hence the condition $0 \in \interior (\dJ(0))$ (understood in $L^\infty$).
 This proves that the regulariser $\reg(\cdot) = \TV(\cdot) + \gamma\norm{\cdot}_1$, $\gamma > 0$, satisfies Assumption~\ref{ass_3}.

Now consider an absolutely one-homogeneous regularisation functional $\reg(\cdot)$ such that $\reg(u) = \reg(u + C)$ for any constant $C$ (for example, $\reg(\cdot) = \TV(\cdot)$).
 By the definition of a convex conjugate, we get the following equality
\begin{equation}\label{dTV(0)}
\reg^*(p) = \sup_u \{(p,u)-\reg(u)\} = \sup_u \{(p,u+C)-\reg(u+C)\} - C(p,\one) = \reg^*(p) - C(p,\one).
\end{equation}

Equality~\eqref{dTV(0)} implies that either $\reg^*(p) = \infty$ or $(p,\one)=0$.
 On the other hand, since $\reg(\cdot)$ is absolutely one-homogeneous, $\reg^*(\cdot) = \charf_{\dJ(0)}(\cdot)$.
 Therefore, $\forall p \in \dJ(0)$ we have that $(p,\one) = 0$, which implies that $\interior (dJ(0)) = \emptyset$.

\section{Model Manifolds}\label{model_manifolds}
Here we derive model manifolds in $L^1$ related to the $L^1$ norm and the $\TV$-seminorm, as well as their combination, i.e. the $\BV$ norm.

\paragraph{The $L^1$ norm.} By the absolute one-homogeneity of the $L^1$ norm we can express its subdifferential at $u \in L^1$ by the following expression
\begin{align}\label{eq:subdiff_l1}
 \partial \norm{u}_1 = \{r \in L^\infty ~|~ \norm{u}_1 = (r,u), \norm{v}_1 \geq (r,v) \forall v \in L^1 \}.
\end{align}
From this one can easily derive that 
\begin{align}\label{eq:positivity_integrand}
\norm{r}_\infty \leq 1 \quad \text{and} \quad \int_\Omega \abs{u} - r u \,dx = 0.
\end{align}
Obviously, for any two numbers $\eta,\nu \in \R$ such that $\abs{\nu} \leq 1$ we have that $\abs{\eta} - \nu \eta \geq 0$, with equality holding if and only if either of the numbers is zero or $\nu = \sign(\eta)$. 
Hence the integrand in \eqref{eq:positivity_integrand} has to be nonnegative for a.e. $x \in \Omega$, and the vanishing integral implies that $\abs{u} = r u$ a.e.
Since $\abs{r} \leq 1$, we get that $r = \sign(u)$ on every set of nonzero measure where $u \neq 0$. 

Consider now $r \in \partial \norm{u}_1$, and assume that $\abs{r} < 1$ whenever $u$ vanishes (in the a.e. sense).
We need to compute the related model manifold 
\begin{align*}
\Mcal = \{v \in L^1 ~|~ \norm{v}_1 - (r,v) = \int_\Omega \abs{v} -  r v = 0 \},
\end{align*}
where $D_{\norm{\cdot}_1}^r(v,u) = \norm{v}_1 - (r,v)$ is the (generalised) Bregman distance.
By the same argument as before, the integrand has to be positive a.e. and the integral may only vanish if the integrand vanishes for a.e. $x \in \Omega$.
Under the above assumptions on $r$, whenever $u$ vanishes we have that $|r|<1$ and therefore $v$ also has to vanish, and whenever $u \neq 0$, $v$ can be any positive number with the same sign as $u$. 
In sum, we get the following expression for the model manifold
\begin{align*}
 \Mcal = \{ v \in L^1 ~|~ \supp(v) \subset \supp(u), v = \lambda \sign(u) \text{ whenever } u \neq 0, \lambda \geq 0  \},
\end{align*}
where $\supp$ has to be interpreted in the a.e. sense.

\paragraph{The $\TV$ seminorm.} The argumentation for $\TV$ is almost identical.
A function $u \in L^1$ lies in $\BV(\Omega)$ if and only if its distributional derivative $\Der u$ is a finite Radon measure~\cite{Ambrosio}, and 
\begin{align*}
 \| \Der u \|_\Radon = \TV (u), 
\end{align*}
where $\| \cdot \|_\Radon$ denotes the Radon norm. 
Furthermore, $\Der u$ possesses a polar decomposition $\Der u = f_{\Der u} |\Der u|$, where $|\Der u |$ denotes the total variation of $u$ and $|f_{\Der u}| = 1$ $|\Der u|$-a.e for the related density function $f_{\Der u}$.
By the chain rule, $p \in \partial \TV(u)$ if and only if $p = \Der^*q$ for some $q$ inb the subdifferential of the $L^1$ norm.
Assuming that $q \in C_0(\Omega,\R^n)$, i.e. $q$ lies in the predual instead of only in the dual space of the space of finite Radon measures, we know by the computations in~\cite{Rasch_PET-MRI} that $q = f_{\Der u}$ $|\Der u|$-a.e. 
In other words, the ``vector field'' $q$ of the subgradient may be decomposed into direction and magnitude, indicating the direction and the hight of jumps across the edges.

We may now rewrite 
\begin{align*}
 D_{\TV}^p(v,u) = D_{\|\cdot\|}^q(\Der v, \Der u) = \int_\Omega 1 - q \cdot f_{\Der v} \, \mathrm{d} |\Der v|.
\end{align*}
This, by the same argument as for $\Lone$, vanishes if and only if either $|\Der v| = 0$ or $f_{\Der v} = q = f_{\Der u}$, implying that the jump set and its direction of $v$ has to be contained in the jump set of $u$. 
In other words, $v$ may only jump where $u$ jumps as well. 
Note that in this case the magnitude of the jump, i.e. $|\Der v|$, can be arbitrary.

It should be mentioned that the (technical) assumption of a continuous subgradient is crucial for this illustration, but not necessarily for the result, meaning that a zero Bregman distance with respect to $\TV$ is still well-defined if $q$ is not continuous. 
However, then it is hard to say anything about the behavior, and we refer to~\cite{Rasch_PET-MRI} for further information.
Moreover, it is possible that $|q| = 1$ on a $|\Der u|$-zero set, i.e. the vector field of the subgradient might be saturated where $u$ is constant. 
In this case, indeed it is possible for $v$ to jump even though $u$ is flat. 
In practice, however, this situation is rarely observed, or has to be enforced by assumption as in the above $\Lone$ case.

\paragraph{The $\BV$ norm.} 
Let $J(u) = \TV(u) + \gamma  \|u \|_1$ for $\gamma > 0$.
Since $\norm{\cdot}_1$ is continuous everywhere on $L^1$ we have that~\cite[Theorem 4.4.3 and Lemma 4.3.1]{Borwein_Zhu}
\begin{align*}
 \mathrm{cont} ( \norm{\cdot}_1 ) \cap \mathrm{dom} (\TV) \neq \emptyset,
\end{align*}
and 
\begin{align*}
 \dJ(v) = \partial \TV(v) + \gamma \partial \norm{u}_1.
\end{align*}
Hence for $p \in \dJ(v)$ we have that there exists $q \in \partial \TV(v)$ and $r \in \partial \norm{u}_1$ such that 
$p = q + \gamma r$.
Considering the Bregman distance $D_J^p(u,v)$, we get the following expression
\begin{align*}
 D_J^p(u,v) 
 &= \TV(u) + \gamma \|u\|_1 - (q + \gamma r, u) \\
 &= \TV(u) - (q,u) + \gamma ( \norm{u}_1 - (r,u) )\\
 &= D_{\TV}^q(u,v) + \gamma D_{\norm{\cdot}_1}^r (u,v).
\end{align*}
Since both Bregman distances are non-negative we have that 
\begin{align*}
 D_J^p(u,v) = 0 \quad \Leftrightarrow \quad D_{\TV}^q(u,v) = 0 \quad \text{and} \quad D_{\norm{\cdot}_1}^r (u,v)= 0,
\end{align*}
implying that the manifold with respect to $J$ and $p$ contains all elements sharing the same jump set (including its direction) and the same (signed) support.

\printbibliography

\end{document}